\documentclass[a4paper,12pt]{article}
\usepackage[T1]{fontenc}
\usepackage[utf8]{inputenc}
\usepackage{mathtools}
\usepackage[english]{babel}
\usepackage{amsthm}
\usepackage{amssymb}
\usepackage{geometry}
\usepackage{comment}
\DeclareMathAlphabet{\pazocal}{OMS}{zplm}{m}{n}
\newcommand{\R}{\mathbb{R}}

\newcommand{\s}{\mathbb{S}}
\newcommand{\D}{\mathbb{D}}

\newcommand{\C}{\pazocal{C}}

\newcommand{\Imm}{\mathrm{Im}}

\newcommand{\h}{\mathbb{H}}

\newcommand{\Vol}{\mathrm{V} \mathrm{o} \mathrm{l}}

\newcommand{\Diam}{\mathrm{D} \mathrm{i} \mathrm{a} \mathrm{m}}

\newcommand{\ein}{\mathrm{E} \mathrm{i} \mathrm{n}}
\newcommand{\so}{\mathrm{S} \mathrm{O}}

\newcommand{\Ker}{\mathrm{K} \mathrm{e} \mathrm{r}}

\newcommand{\Cc}{\mathcal{C}}
\newlength\tindent
\renewcommand{\indent}{\hspace*{\tindent}}


\theoremstyle{definition}
\newtheorem{rem}{Remark}[section]
\newtheorem{defn}[rem]{Definition}

\newtheorem{exmp}[rem]{Example}
\theoremstyle{plain}
\newtheorem{theo}[rem]{Theorem}
\newtheorem{prop}[rem]{Proposition}
\newtheorem{coro}[rem]{Corollary}
\newtheorem{lemma}[rem]{Lemma}
\newtheorem*{idefn}{Definition}%
\newtheorem*{iprop}{Proposition}%
\newtheorem*{itheo}{Theorem}%

\usepackage{lipsum}

\begin{document}

\title{Global hyperbolicity in higher signature}
\author{Roméo Troubat}
\maketitle

\begin{abstract}
\noindent We provide a generalization of global hyperbolicity in pseudo-Riemannian spaces of signature $(p,q)$ for $p \geqslant q \geqslant 2$. We generalize a number of well known results of globally hyperbolic Lorentzian spaces to their pseudo-Riemannian counterpart, including a result of topological structure and a result on the pseudo-Riemannian Plateau problem. Finally, we generalise a result by Mess by showing that the representations in $\so_0(p,q+1)$ stabilizing an achronal non-purely lightlike $(p-1)$-sphere in $\partial \mathbb{H}^{p,q}$ are exactly the holonomies of maximal globally hyperbolic pseudo-Riemannian spacetimes of constant curvature $-1$ with complete Cauchy surfaces and convex timelike submanifolds.
\end{abstract}

\tableofcontents

\section*{Introduction}

The study of Lorentzian geometry on smooth manifolds often takes place through the study the causal paths, given that a Lorentzian manifold $M$ has a suitable notion of causality. The strongest notion of causality one may ask is that $M$ is \emph{globally hyperbolic}. When $M$ is strongly causal, it was shown by Geroch in \cite{Gerochsplittingtheorem} that being globally hyperbolic can be defined by three equivalent definition. The first one is the existence of a \emph{Cauchy surface}, a subset of $M$ through which every inextendible causal path passes exactly once. The second one is the existence of a \emph{Cauchy time function}, a continuous map from $M$ to $\R$ which restricts into a bijection on any inextendible causal path. The third one is the compactness of the set of causal paths between any two points up to reparametrization. 

Global hyperbolicity is a very convenient setting to work in; it guarantees the existence of causal geodesics between any pair of causally related points, it implies strong topological result on the spacetime, and other interesting properties. \\

Instead of a Lorentzian metric, one may want to consider a metric of signature $(p,q)$ where $\min(p,q) \geqslant 2$, which we shall call a higher signature pseudo-Riemannian metric. Geometry in higher signature has been less studied than its Lorentzian counterpart ; one reason for this is the lack of physical interpretation. Most of the tools developped for Lorentzian geometry fail in this new setting, mostly because of the connectedness the set $(g_{p,q} < 0)$ in $(\R^{p,q}, g_{p,q})$ which does not yield a notion of time orientation for smooth paths. \\

Assume $(M,g)$ is a pseudo-Riemannian space of signature $(p,q)$, let $N_0$ be a smooth manifold of dimension $q$ and let $f : N_0 \rightarrow M$ be an immersion from $N_0$ to $M$. We will say that $f$ is timelike (resp. causal) when for each $x$ in $N_0$, the restriction of the metric to the image of $df_x$ is negative (resp. non-positive). Like in Lorentzian geometry, we show in section 3.3 that this notion can be extended to continuous maps from $N_0$ to $M$ which we will call \emph{causal maps}. We will say that such a map is \emph{inextendible} if for any path $c : I \rightarrow N_0$ inextendible in $N_0$ (meaning not converging in its extremities), the path $f \circ c$ is inextendible in $M$. With those new definitions, one may define a notion of Cauchy surfaces in higher signature.

\begin{idefn}
    A subset $S$ of $M$ is a \emph{Cauchy surface} if for any inextendible causal map $f : N_0 \rightarrow M$, there exists a unique $x$ in $N_0$ such that $f(x) \in S$.
\end{idefn}

The generalization of Cauchy time functions also becomes quite clear.

\begin{idefn}
A continuous map $T : M \rightarrow N$ where $N$ is a $q$-dimensional manifold is a \emph{Cauchy time function} if for any inextendible causal map $f : N_0 \rightarrow M$, the map $T \circ f : N_0 \rightarrow N$ is a homeomorphism.
\end{idefn}

We give a main class of examples of spacetimes admitting a Cauchy time function.

\begin{iprop}
    Let $T : (M^{p+q}, g_M) \rightarrow (N^q, g_N)$ be a complete Riemannian submersion where $N$ is simply connected and let $g$ be the $(p,q)$ metric defined on $M$ by the splitting $Ker(d T) \oplus (Ker(d T))^{\perp}$ where each fiber is positive and its orthogonal is negative, i.e $g = g_M - 2 T^* g_N$. Then the map $T : M \rightarrow N$ is a Cauchy time function.
\end{iprop}

In particular, while Lorentzian globally hyperbolic spacetimes are always homeomorphic to products $S \times \R$ where $S$ is any Cauchy surface, there exists spacetimes $M$ admitting a Cauchy time function $T : M \rightarrow N$ where $T$ is any fiber bundle with simply connected basis. We prove two topological results on spacetimes admitting a Cauchy time function akin to the one given by Geroch in the Lorentzian case.

\begin{theo}\label{theo:topostructure1}
    Let $M$ be a spacetime with Cauchy time function $T : M \rightarrow N$, assume that $M$ admits a foliation $F$ by causal spaces of dimension $q$ and let $P$ be a level set of $T$. Then there exists a homeomorphism $M \simeq P \times N$ such that that $T$ is the projection $pr_2$ on the second factor and the leaves of $F$ are the fibers of the projection $pr_1$ on the first factor.
\end{theo}

\begin{theo}\label{theo:topostructure2}
Let $M$ be a spacetime with a smooth Cauchy time function $T : M \rightarrow N$ having at least one compact level set. Let $\gamma : I \rightarrow N$ be a smooth path. Then the Lorentzian space $\gamma^* M = \bigsqcup_{t \in I} T^{-1}(\gamma(t))$ is globally hyperbolic with the projection on $I$ as a Cauchy time function. In particular, $T$ is a fiber bundle on $N$.
\end{theo}

By studying examples, we will observe that the mere existence of a Cauchy surface is too weak an hypothesis in the non-Lorentzian case as it does not give any information on the global structure of $M$. While the existence of a Cauchy time function on $M$ does give very strong informations on the spacetime, it is in practrice quite hard to show that a given spacetime admits such a Cauchy time function. The generalization of global hyperbolicity which we will choose is based on the third definition in Lorentzian geometry, that of the compactness of causal diamond. Let $S$ be a $(q-1)$-sphere in $M$ which is included in an inextendible causal map. Let $\C(S)$ be the set of maps $f : \overline{\D}^q$ which are embeddings with image $S$ on $\partial \D^q$ and causal on $\D^q$ up to reparametrization of $\D^q$.

\begin{idefn}
We'll say that $M$ is \emph{globally hyperbolic} it satisfies those two conditions :

\begin{enumerate}
    \item There exists a $q$-dimensional manifold $N$ such that for any inextendible causal map $f : N_0 \rightarrow M$, $N_0$ is homeomorphic to $N$.
    \item For any $(q-1)$-sphere $S$ in $M$ which is included in an inextendible causal map, the set $\C(S)$ is compact for the uniform topology.
\end{enumerate}
\end{idefn}

The tree of implications between the three notions is in Figure \ref{intro:implications}.

\begin{figure}[h]
    \begin{center}
      \includegraphics[width=.7\linewidth]{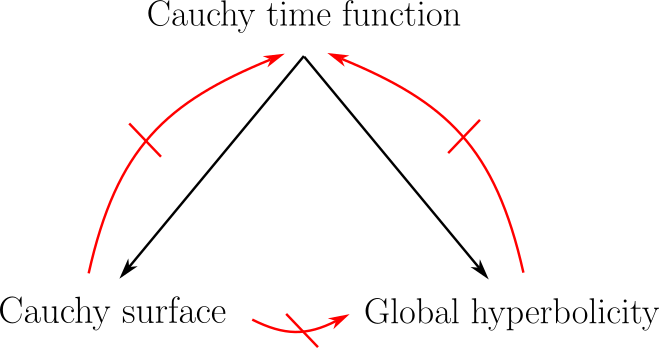}\label{intro:implications}
      \end{center}
      \caption[]{The tree of implications}
    \end{figure}

We show that globally hyperbolic spacetimes verify some interesting properties inspired by the Lorentzian case; In a Lorentzian globally hyperbolic spacetime, any two causally related points admit a causal geodesic between them. While this is no longer true in higher signature, we prove a similar result :

\begin{theo}\label{theo:plateauproblem}
Let $M$ be a globally hyperbolic spacetime and $S$ a $(q-1)$-sphere in $M$ which is included in an inextendible causal map. Then there exists an element in $\C(S)$ of maximal timelike volume.
\end{theo}

In the last section of the article, we give an application of the notions we introduced to the study of an interesting class of representations in $\so_0(p,q+1)$. A discrete and faithful representation $\rho : \Gamma \rightarrow \so_0(p,q)$ is said to be \emph{GH-regular} if it stabilizes a maximal spacelike submanifold of dimension $p$ in $\h^{p,q}$. When the action is cocompact, those representations are called either \emph{spacelike cocompact} (Beyrer-Kassel, \cite{beyrer2023mathbbhpqconvex}), \emph{GHC-regular} (Barbot, \cite{barbot2013deformations}) or $\h^{p,q}$-convex cocompact when $\Gamma$ is hyperbolic (Danciger-Guéritaud-Kassel, \cite{danciger2017convex}, \cite{DGK17}). Those representations are of particular interest as it was proven by Barbot for $q=1$ and Beyrer-Kassel for $q \geq 2$ that they formed a union of connected components in the character variety of $\Gamma$ in $\so_0(p, q+1)$. \\

Given a GH-regular representation $\rho : \Gamma \rightarrow \so_0(p, q+1)$, one may associate a $(\so_0(p, q+1), \h^{p,q})$-structure $\Omega(\Lambda)/\rho$. For $q=1$, it was proven by Mess that those structures characterised GH-regular representations.

\begin{itheo}[Mess, \cite{mess2007lorentz}]
    Let $\rho : \Gamma \rightarrow \so_0(p, 2)$ be a GH-regular representation. Then the space $\Omega(\Lambda)/\rho$ is a maximal globally hyperbolic $(\so_0(p, 2), \h^{p,1})$-structure with complete Cauchy surfaces. Inversely, any maximal globally hyperbolic $(\so_0(p, 2), \h^{p,1})$-structure $M$ with a complete Cauchy surface has holonomy $\rho : \pi_1(M) \rightarrow \so_0(p,2)$ a GH-regular representation and is isometric to $\Omega(\Lambda)/\rho$.
\end{itheo}

We generalize this result to the case $q \geq 2$.

\begin{theo}\label{theo:messgeneral}
    Let $\rho : \Gamma \rightarrow \so_0(p, q+1)$ be a GH-regular representation. Then the space $\Omega(\Lambda)/\rho$ is a maximal globally hyperbolic time-convex $(\so_0(p, q+1), \h^{p,q})$-structure with complete Cauchy surfaces. Inversely, any maximal globally hyperbolic time-convex $(\so_0(p, q+1), \h^{p,q})$-structure $M$ with a complete Cauchy surface has holonomy $\rho : \pi_1(M) \rightarrow \so_0(p,q+1)$ a GH-regular representation and is isometric to $\Omega(\Lambda)/\rho$.
\end{theo}

The case where the Cauchy surfaces are compact corresponds to the case where $\rho$ is spacelike cocompact. 

In the first section of this paper, we will recall some known properties of globally hyperbolic spacetimes and of Lorentzian geometry in general. In section 2, we will generalize the notion of causality to the higher signature setting and we will study the local structure of causal maps. In section 3, we will define the notions of Cauchy surfaces, Cauchy time functions and global hyperbolicity and prove some interesting results on them. We will in particular prove Theorems \ref{theo:topostructure1}, \ref{theo:topostructure2} and \ref{theo:plateauproblem}. In section 4, we will study GH-regular representations in $\so_0(p, q+1)$ and prove Theorem \ref{theo:messgeneral}.

\section{Causality in Lorentzian manifolds}

\subsection{Lorentzian manifolds and time orientation}

Let $M$ be a smooth manifold of dimension $p+1$ and $(g_x)_{x \in M}$ a smooth familly of quadratic forms of signature $(p,1)$ on $T_x M, x \in M$. 
The couple $(M,g)$ is said to be a \emph{Lorentzian manifold}. 
Given a point $x \in M$, a tangent vector $v \in T_x M$ can be differentiated by the sign of its norm.

\begin{defn}
    A vector $v \in T_x M$ is said to be 

    \begin{itemize}
        \item \emph{spacelike} when $g(v) > 0$, 
        \item \emph{lightlike} when $g(v) = 0$, 
        \item \emph{timelike} when $g(v) < 0$.
    \end{itemize}

    When $g(v)$ is lightlike or timelike, $v$ is said to be \emph{causal}.
\end{defn}

We can extend those infinitesimal definitions to a smooth path of $M$.

\begin{defn}
    Let $c : I \subset \R \rightarrow M$ be a smooth path. It is said to be

    \begin{itemize}
        \item \emph{spacelike} if for all $t \in I$, $g(\dot{c}(t)) >0$,
        \item \emph{lightlike} if for all $t \in I$, $g(\dot{c}(t)) =0$,
        \item \emph{timelike} if for all $t \in I$, $g(\dot{c}(t)) <0$.
    \end{itemize}

    When for all $t \in I$, $g(\dot{c}(t))$ is causal, $c$ is said to be \emph{causal}.
\end{defn}

Let $x$ be a point of $M$ and let $\mathbb{P}^{<0}(T_x M)$ be the subset of $\mathbb{P}(T_x M)$ made of the timelike lines of $T_x M$. This space can be identified with the hyperbolic space ; in particular, it is contractible. The set of oriented timelike rays $\s^{<0}(T_x M)$ is a double covering of $\mathbb{P}^{<0}(T_x M)$, therefore it must be the disjoint union of two connected components, both homeomorphic to $\mathbb{P}^{<0}(T_x M)$.

\begin{figure}[h]
    \begin{center}
      \includegraphics[width=.7\linewidth]{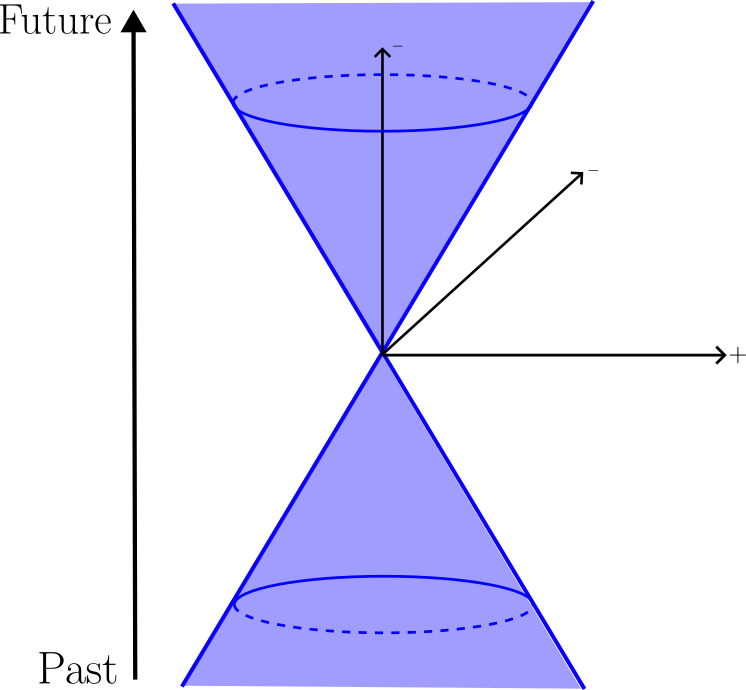}
      \end{center}
      \caption[]{The causal vectors in the lorentzian space $\R^{2,1}$}
    \end{figure}

\begin{defn}
    The choice of one of those components gives a \emph{local time orientation} of $M$ in $x$. The chosen component will be denoted by $\s^{<0}(T_x M)^+$ whereas the opposite component will be denote by $\s^{<0}(T_x M)^-$. A tangent timelike ray in $x$ will be said to be \emph{future-oriented} if it is contained in $\s^{<0}(T_x M)^+$. Otherwise, it will be \emph{past-oriented}. A tangent timelike vector $v \in T_x M$ will be future (resp. past) oriented if $[v] \in \s^{<0}(T_x M)^+$ (resp. $[v] \in \s^{<0}(T_x M)^-$).
\end{defn}

 \begin{rem}
    This notion can be extended to causal tangent rays ; for such a ray $[v] \in \s(T_x M)$, we'll say that $[v]$ is future (resp. past) oriented if it is in the closure of $\s^{<0}(T_x M)^+$ (resp. $\s^{<0}(T_x M)^-$) in $\s(T_x M)$. This also applies to non-zero tangent causal vectors.
 \end{rem}

This time orientation can be extended to the whole manifold $M$, provided that the local 
 orientation in each point stays coherent on the whole manifold.

\begin{defn}
    When the fiber bundle $\s^{<0}(TM)$ over $M$ is trivial, $M$ is said to be $\emph{time orientable}$. The choice of one of its two connected components gives a \emph{time orientation} of $M$. 
    We'll denote by $\s^{<0}(TM)^+$ (resp. $\s^{<0}(TM)^-$) the chosen component, called the bundle of \emph{future-oriented} (resp. \emph{past-oriented}) tangent rays. 
\end{defn}

\begin{defn}
    A causal path $c$ in a time oriented manifold $M$ is said to be future (resp. past) oriented if for all $t$, $\dot{c}(t)$ is future (resp. past) oriented.
\end{defn}

A time orientation on a lorentzian manifold is sometimes defined by a nowhere vanishing 
causal vector field $X$ on $M$. Those two definitions are in fact equivalent.

\begin{prop}\label{timelikevectorfield}
    Let $M$ be a time oriented lorentzian manifold. There exists a causal future-oriented vector field $X$ on $M$.
\end{prop}

\begin{proof}
    Each fiber of the bundle $\s^{<0}(TM)^+$ is contractible, meaning there must exist a smooth section $X : M \rightarrow \s^{<0}(TM)^+$ (see \cite{topologyoffiberbundles}), hence the result.
\end{proof}

\begin{rem}
    Let $M$ be a smooth manifold admitting a smooth nowhere-vanishing vector field $X$. Let $\overline{g}$ be a Riemannian metric on $M$ and for each $x \in M$ and $v \in T_x M$, let

    \[g(v) = \overline{g}(\pi_{X(x)^{\perp}}(v)) - \overline{g}(\pi_{X(x)}(v)),\]

    where $\pi_E$ is the orthogonal projection with image $E$. This defines a Lorentzian metric $g$ on $M$ for which $X$ is timelike, thus giving us the following result : The smooth manifolds admitting a Lorentzian metric are exactly the manifolds admitting a smooth non-vanishing vector field.
\end{rem}

\begin{coro}
    Since each smooth vector field is integrable, every time oriented Lorentzian manifold admits a smooth future-oriented one dimensional timelike foliation.
\end{coro}

Mathematicians and physicists alike often assume that $M$ is both orientable and time-orientable, 
as it gives both a reasonable model for our spacetime when $p=3$ and a convenient context in which to do Lorentzian geometry.

\begin{defn}
    A Lorentzian manifold $M$ which is both orientable and time orientable is called a \emph{spacetime}.
\end{defn}

Unless specified otherwise, every Lorentzian manifold we will consider in this paper will be a Lorentzian spacetime.

\begin{comment}

It is possible to extend the notion of future-oriented path to piece-wise smooth paths, as long as every smooth piece keeps the same orientation.

\begin{defn}
    A piece-wise smooth path $c : I \rightarrow M$ is said to be future-oriented if for each $t \in I$ such that $c$ is differentiable in $t$, $\dot{c}(t)$ is future-oriented.
\end{defn}

\begin{prop}
    Let $c : I \rightarrow M$ be a piece-wise smooth future-oriented path and let $J$ be an closed subset of $I$ such that each point where $c$ is not smooth is contained in $I \setminus J$. There exists a smooth future-oriented path $\tilde{c} : I \rightarrow M$ such that for each $t \in J$, $\tilde{c}(t) = c(t)$.
\end{prop}

This result tells us that every piece-wise smooth future-oriented causal path can be approximated by a smooth future-oriented causal path.

\end{comment}

\subsection{Causality in a Lorentzian spacetime}

Now that we have defined what it means for a path to be causal, for $x \in M$, we can consider the set of points which can be causally related to $x$.

\begin{defn}
    The set $I^+(x)$ (resp. $J^+(x)$) is the set of points $y \in M$ such that there exists a timelike (resp causal) future-oriented path $c$ coming from $x$ to $y$. 
\end{defn}

It is possible to show that the notion of causality is locally linked to the nature of the geodesics coming from a point $x$ of $M$.

\begin{defn}
    Let $x \in M$. A normal neighborhood of $x$ in $M$ is an open subset $U$ of $T_x M$ containing $0$ on which $exp : T_x M \rightarrow M$ is a diffeomorphism. The sets $U$ and $exp(U)$ will sometimes be identified.
\end{defn}

\begin{prop}\label{Lorentzianlocalstructure}
    Let $U \subset T_x M$ be a normal subset of $M$ centered on $x$ and let $c : I \rightarrow U$ be a timelike future-oriented path such that $c(0) = 0$. Then for all $t >0$ (resp. $t < 0$), we have $g_x(c(t)) < 0$ and $c(t)$ is in the future (resp. past) time cone $g_x < 0$ of $T_x M$.
\end{prop}

\begin{rem}
    The same result can be stated for causal paths ; the points $y \in U$ such that $g_x(y) = 0$ are exactly those linked to $0$ by a lightlike geodesic.
\end{rem}

This describes the local structure of a smooth future-oriented causal path and allows us to weaken the regularity condition on $c$. We are thus able to consider paths which are only continuous but locally verify a property analogous to \ref{Lorentzianlocalstructure}.

\begin{defn}
    A continous path $c : I \rightarrow M$ is said to be timelike (resp. causal) if for each $t_0 \in I$, there exists a normal neighborhood $U$ of $c(t_0)$ such that for each $t < t_0$ (resp. $t > t_0$) with $c(t) \in U$, the geodesic $[c(t_0), c(t)]$ is causal past-oriented (resp. future-oriented).
\end{defn}

It is possible to show that the set of smooth causal paths is dense in the set of continous causal paths for the uniform topology.

\begin{prop}
    Let $c : I \rightarrow M$ be a continuous causal path and $d$ a Riemannian distance on $M$. For each positive continuous map $\delta : I \rightarrow \R$, there exists a smooth causal path $\overline{c} : I \rightarrow M$ such that for all $t \in I$, $d(c(t), \overline{c}(t)) \leqslant \delta(t)$.
\end{prop}

In particular, piecewise-smooth causal paths are continuous causal paths in this sense, 
as long as every smooth piece keeps the same time orientation.

\begin{coro}
    Let $c_1 : (a,b) \rightarrow M$ and $c_2 : (b,c) \rightarrow M$ be two smooth timelike (resp. causal) paths in $M$ 
    such that $c_1(b) = c_2(b)$ and $\dot{c_1}(b)$ and $\dot{c_2}(b)$ are both in the same connected component of $\{g_{c_1(b)} < 0\}$. 
    Then the concatenation $c$ of $c_1$ and $c_2$ is a continuous timelike (resp. causal) path.
\end{coro}

\begin{proof}
    By using proposition \ref{Lorentzianlocalstructure}, this tells us that for $\varepsilon$ small enough and for each $b- \varepsilon < t < b$, 
    $c_1(t)$ is in the past of $c_1(b)$ and for each $b + \varepsilon > t > b$, $c_2(t)$ is in the future of $c_2(b)$, hence the result.
\end{proof}

In what follows, unless specified otherwise, a future-oriented causal path will only be assumed to be continuous. As Lorentzian geometry was primarily developped as a tool to modelize general relativity, we would like to give conditions on spacetimes to get a physically relevant notion of causality. One of those obvious conditions is that we do not want a point in a spacetime $M$ to be causally related to itself.

\begin{defn}
    A spacetime $M$ is said to be \emph{non-causal} if there exists a point $x$ and a non-trivial causal path $c : [a,b] \rightarrow M$ such that $c(a) = c(b) = x$. Alternativaly, $M$ is non-causal if there exists a future-oriented causal embedding $c : \s^1 \rightarrow M$. When $M$ is non-non-causal, it is said to be \emph{causal}.
\end{defn}

Here are some examples of non-causal spacetimes.

\begin{exmp}
    \begin{itemize}
        \item The Einstein space, $\ein^{n,1} \simeq (\s^n \times \s^1, ds_n^2 - ds_1^2)$ where $ds_n^2$ is the round metric on $\s^n$. The map $t \in \s^1 \mapsto (p, t)$ is a causal loop.
        \item Any space of the form $(M \times \s^1, ds^2 - ds_1^2)$ where $ds^2$ is a Riemannian metric on $M$, for the same reason.
        \item It is possible to show that any compact spacetime is non-causal (see \cite{largescalestructure}).
    \end{itemize}
\end{exmp}

One can reinforce this notion of causality with the existence of a time function.

\begin{defn}
    A continuous map $T : M \rightarrow \R$ is said to be a \emph{time function} if for each future causal curve $c : I \rightarrow M$, $T \circ c : I \rightarrow \R$ is increasing.
\end{defn}

Such a map gives us a way to parametrize all timelike curves by intervals of $\R$. In particular, the existence 
of a time function on $M$ is a stronger causality condition than that of causality.

\begin{coro}
    If $M$ admits a time function, $M$ must be causal.
\end{coro}

\begin{proof}
    Assume there exists a future-oriented causal path $c : (a,b) \rightarrow M$ such that $c(a) = c(b)$. 
    Then we must have $T(c(a)) = T(c(b))$, which is impossible since $T$ is assumed to be increasing on every causal curve.
\end{proof}

\begin{defn}\label{inextendiblepath}
    A path $c : (a,b) \rightarrow M$ is said to be \emph{extendible} if $c$ converges in $M$ near $a$ or $b$. Otherwise, $c$ is \emph{inextendible}.
\end{defn}

As we've seen, it is quite easy to construct smooth future otiented paths in a spacetime $M$, thus it is also possible to extend an extendible future-oriented path into an inextendible future path.

\begin{prop}
    Let $c : (a,b) \rightarrow M$ be a future path which is extendible in $b$. There exists an inextendible future path $\tilde{c} : I \rightarrow M$ such that $(a,b) \subset I$ and for all $t \in (a,b)$, $\tilde{c}(t) = c(t)$.
\end{prop}

\begin{proof}
    The path $c$ can be extended indefinitely by taking a normal neighborhood centered on its endpoint, hence the result. When $c$ is smooth, it is also possible to extend $\dot{c}(t), t \in \overline{I}$ into a causal vector field on $M$ and to take the flow line containing $c$.
\end{proof}

\subsection{Cauchy surfaces and global hyperbolicity}

We will now give a stronger notion of causality on a spacetime $M$. This notion detects both the fact that $M$ is causal, and the absence of causal paths which go to the boundary of the spacetime $M$ in finite time.

\begin{defn}
    Let $x,y \in M$. We will denote by $\Cc(x,y)$ the set of future-oriented causal paths coming from $x$ to $y$, which we will endowe with the uniform $C^0$ topology, meaning that the open set of $\Cc(x,y)$ will be the future-oriented paths from $x$ to $y$ entirely contained in an open set $U$ of $M$.
\end{defn}

\begin{defn}
    A spacetime $M$ is said to be \emph{globally hyperbolic} if for each $x,y \in M$, $\Cc(x,y)$ is compact.
\end{defn}

This is actually a very strong condition, which as we will see, implies both the causality of $M$ and the existence of a time function.

\begin{defn}
    A subset $S \subset M$ is a \emph{Cauchy surface} of $M$ if for each inextendible future causal curve $c : I \rightarrow M$, there exists a unique $t \in I$ such that $c(t) \in S$.
\end{defn}

\begin{figure}[hbt]
    \begin{center}
      \includegraphics[width=.7\linewidth]{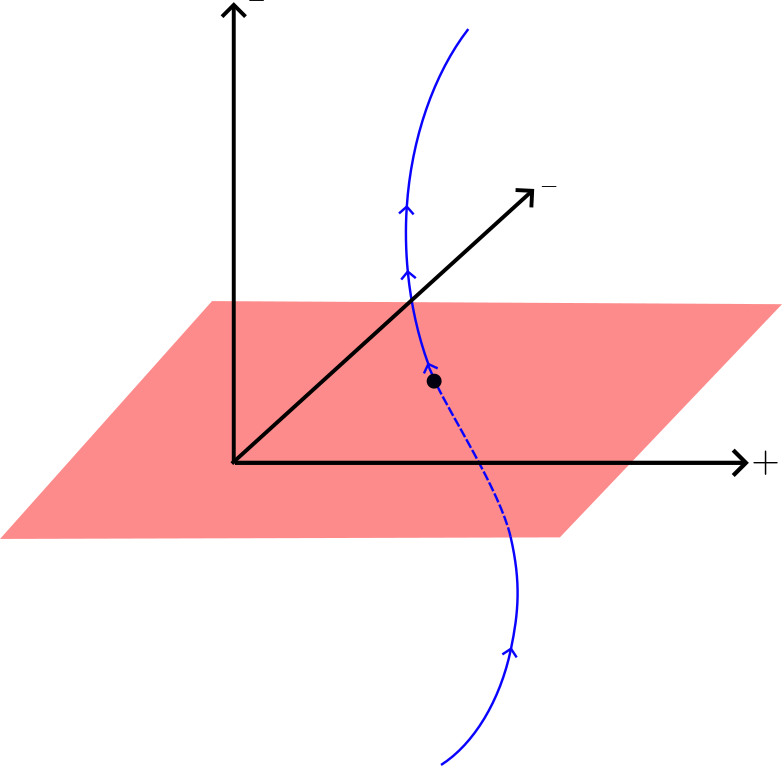}
      \end{center}
      \caption[]{A Cauchy surface of $\R^{2,1}$ which an inextendible timelike curve intersects in a unique point.}
    \end{figure}

\begin{defn}
    A map $T : M \rightarrow \R$ is said to be a \emph{Cauchy time function} if for each $t \in \R$, $T^{-1}(t)$ is a Cauchy surface of $M$. Alternatively, this means that for each inextendible future causal curve $c : I \rightarrow M$, $T \circ c : I \rightarrow \R$ is bijective and increasing.
\end{defn}

\begin{rem}
    In particular, a Cauchy time function is a time function.
\end{rem}

The fact that a Cauchy time function $T$ is injective gives us the causality of $M$, and the fact that is must be surjective on each inextendible curve tells us that we cannot find a future causal curve which approches the boundary of the spacetime $M$ without passing by every "time".

\begin{prop}\label{GH}
    The three following properties are equivalent :

    \begin{itemize}
        \item $M$ is globally hyperbolic,
        \item $M$ admits a Cauchy surface,
        \item $M$ admits a Cauchy time function.
    \end{itemize}
\end{prop}

This condition actually implies a strong topological rigidity on $M$.

\begin{prop}\label{topologicalrigidity}
    Let $M$ be a globally hyperbolic space with a Cauchy surface $S$. Then $M$ is homeomorphic to $S \times \R$ where each slice $S \times \{*\}$ is a Cauchy surface and each line $\{*\} \times \R$ is a timelike curve.
\end{prop}

\begin{proof}
    Let $T$ be a Cauchy time function with $S = T^{-1}(0)$ and let $X$ be a smooth timelike vector field on $M$, whose existence is guaranteed by \ref{timelikevectorfield}. For each $x \in M$, let $c_x$ be the flow line containing $x$. It is an inextendible timelike path, meaning it must encounter $S$ in a unique point which we will denote by $\varphi(x)$. Finally, let us define the map

    \[\psi : x \in M \longmapsto (\varphi(x), T(x)) \in S \times \R.\]

    By definition of a Cauchy time function, $\psi$ is a homeomorphism for which the slices $S \times \{t\}$ correspond to the Cauchy surfaces $T^{-1}(t)$ and the slices $\{p\} \times \R$ to the timelike flow lines, hence the result.
\end{proof}

For the rest of this paper, our goal will be to generalize the preceeding notions to pseudo-Riemannian spaces of higher signature.

\section{Causality in higher signature}

\subsection{Time orientation in higher signature}\label{pseudoriemannianexamples}

Let~$M$ be connected a smooth manifold and~$(g_x)_{x \in M}$ a smooth family of quadratic forms of signature~$(p,q)$. The pair~$(M,g)$ is said to be a \emph{pseudo-Riemannian manifold} of signature~$(p,q)$.

\begin{rem}
    Taking~$q = 1$ gives us a Lorentzian geometry on~$M$.
\end{rem}

\noindent As this is a generalization of Lorentzian geometry, we will use the same terminology.

\begin{defn}
    A vector~$v \in T_x M$ is said to be 

    \begin{itemize}
        \item \emph{spacelike} when~$g(v) > 0$, 
        \item \emph{lightlike} when~$g(v) = 0$, 
        \item \emph{causal} when~$g(v) \leqslant 0$,
        \item \emph{timelike} when~$g(v) < 0$.
    \end{itemize}
\end{defn}

\noindent Since~$p$ and~$q$ can both be greater than~$2$, it is also possible to discuss the nature of 
the~$2$-planes in~$T_x M$, or more generally, of the~$k$-planes for~$k \leqslant q$.

\begin{defn}
    A tangent vector space~$Q \subset T_x M$ of dimension~$k$ is said to be 

    \begin{itemize}
        \item \emph{spacelike} if~$g|_Q > 0$, 
        \item \emph{lightlike} if~$g|_Q = 0$, 
        \item \emph{causal} if~$g|_Q \leqslant 0$,
        \item \emph{timelike} if~$g|_Q < 0$.
    \end{itemize}
\end{defn}

\noindent It is possible to extend the notion of time orientability where, instead of taking a coherent orientation of timelike vectors, we take an orientation of timelike tangent vector spaces of dimension~$q$. Let~$x$ be a point in~$M$. The set~$Gr_q^{<0}(T_x M)$ of~$q$-planes of~$T_x M$ on which~$g_x$ is definite negative can be viewed as the Riemannian symmetric space~$\so_0(p,q)/\so(p) \times \so(q)$. In particular, it is contractible. Let~$\widetilde{Gr_q^{<0}(T_x M)}$ be the set of oriented timelike~$q$-planes of~$T_x M$. It is a double covering of~$Gr_q^{<0}(T_x M)$ which is simply connected, therefore it is the disjoint union of two copies of~$Gr_q^{<0}(T_x M)$. 

\begin{defn}
    The choice of one of those components gives us a \emph{local time orientation} of~$M$ at~$x$.
\end{defn}

\noindent This time orientation is called local as it may be extended on a neighborhood of $x$. In the same manner as in the Lorentzian case, we are going to require that this local time orientation 
on the timelike~$q$-vector spaces in~$T_x M$ extends to the whole manifold~$M$ in a coherent fashion.

\begin{defn}
    When the fiber bundle~$\overline{Gr_q^{<0}(TM)}$ is trivial,~$M$ is said to be \emph{time orientable}. The choice of a connected component gives a \emph{time orientation} of~$M$. The chosen component~$Gr_q^{<0}(TM)^+$ is the bundle of \emph{future oriented} tangent~$q$-spaces.
\end{defn}

\noindent From now on, we will assume that~$M$ is time oriented.

\begin{defn}
    An oriented causal tangent~$q$-vector space~$Q \subset T_x M$ is \emph{future oriented} if it is in the closure of~$Gr_q^{<0}(T_x M)^+$ in the Grassmanian of~$q$-vector spaces of~$T_x M$.
\end{defn}

\noindent Taking a time orientation on~$M$ is equivalent to taking a distribution 
of oriented timelike~$q$-vector spaces of~$M$.

\begin{prop}
    Let~$M$ be a time oriented pseudo-Riemannian manifold~$M$ of signature~$(p,q)$. There exists a smooth distribution of oriented timelike tangent vector spaces of~$M$ of dimension~$q$.
\end{prop}

\begin{proof}
    The bundle~$Gr_q^{<0}(TM)^+$ has contractible fibers, meaning it must admit a smooth section (see \cite{topologyoffiberbundles}), hence the result.
\end{proof}

\noindent The existence of a smooth distribution of dimension~$q$ on a smooth manifold~$M^{p+q}$ 
characterizes the existence of a pseudo-Riemannian metric of signature~$(p,q)$ on~$M$.

\begin{prop}
    Let~$M$ be a smooth manifold of dimension~$p+q$.~$M$ admits a smooth distribution of dimension~$q$ if and only if~$M$ admits a smooth pseudo-Riemannian metric.
\end{prop}

\begin{proof}
Assume~$M$ admits a pseudo-Riemannian metric~$g$. The fiber bundle~$Gr_q^{<0}(TM)$ has contractible fibers, thus it must admits a smooth section, hence the result.

    \noindent Conversely, assume that~$M$ admits a smooth distribution~$D$ of dimension~$q$. Let~$\overline{g}$ be a Riemannian metric on~$M$ and let~$x \in M, v \in T_x M$. We will write

    \[g(v) = \overline{g}(\pi_{D(x)^{\perp}}(v)) - \overline{g}(\pi_{D(x)}(v)),\]

    \noindent where~$\pi_E$ is the orthogonal projection of image~$E$. This gives us a metric of signature~$(p,q)$ on~$M$ for which~$D$ is a smooth distribution of timelike tangent spaces, hence the result.
\end{proof}

\begin{rem}
    Contrary to the Lorentzian case, since not every~$q$-dimensional distribution of vector spaces is integrable, this does not give a timelike foliation of~$M$. It was however proven by Thurston in \cite{Thurston2}, \cite{Thurston1974} that every~$q$-dimensional distribution on a manifold~$M$ is always homotopic to an integrable distribution. A natural question one could ask would be whether it is always possible, given a pseudo-Riemannian metric, to find an integrable distribution of dimension~$q$ which is timelike at every point. As we will see later, the answer to this question is negative.
\end{rem}

\begin{defn}
    A pseudo-Riemannian manifold of signature~$(p,q)$ is called a \emph{$(p,q)$-spacetime} if it is both orientable and time orientable.
\end{defn}

\noindent Examples of non time-orientable pseudo-Riemannian manifolds include but are not limited to 
the metric products~$(P \times N, ds_P^2 - ds_N^2)$ where~$N$ is a non-orientable manifold of dimension~$q$. In the rest of this paper, it will be assumed that~$M$ is a~$(p,q)$ spacetime. 
Let us now generalize the notion of causal curve in higher signature. 
As we have given a way to distinguish different~$k$-vector spaces based on their induced signature, it is natural to extend 
the notion of causal curve to the whole family of immersions of dimension~$k$ for~$k \leqslant q$.

\begin{defn}
    An immersion of dimension~$k \leqslant q$,~$f \colon N_0^k \rightarrow M$ is said to be \emph{timelike} (resp. \emph{causal}, \emph{lightlike}, \emph{spacelike}) if for each~$x \in N_0$, the image of~$d_x f$ in~$T_{f(x)} M$ is timelike (resp. causal, lightlike, spacelike). 
\end{defn}

\noindent Alternatively, one could define~$f \colon N_0 \rightarrow M$ to be timelike (resp. causal) if the pullback metric~$f^* g$ on~$N_0$ is negative (resp. non-positive). 
When $M$ is furthermore time oriented, it is possible to ask for an immersion of dimension~$q$ to be future-oriented or past-oriented.

\begin{defn}
    Let~$N_0$ be an oriented space of dimension~$q$. A causal immersion~$f \colon N_0 \rightarrow M$ is future-oriented if for each~$x \in N_0$,~$\Imm(df)$ with the orientation induced by~$N_0$ is future-oriented.
\end{defn}

\noindent In fact, the existence of a time orientation on~$M$ insures us that any causal immersion of dimension~$q$ is orientable 
in a way that makes it future-oriented.

\begin{lemma}
    Let~$f \colon N_0 \rightarrow M$ be a causal immersion of dimension~$q$ in a spacetime~$M$. Then the manifold~$N_0$ is orientable and one of its orientations makes~$f$ future-oriented.
\end{lemma}

\begin{proof}
    It suffices to take the pullback orientation of the time-orientation of~$M$ on~$N_0$ via~$f$. This also makes~$f$ future-oriented.
\end{proof}

\noindent Finally, note that the notion of time orientation only depends on the conformal class of the manifold $M$. Let us exhibit some examples of commonly encountered pseudo-Riemannian manifolds by generalizing those given in the Lorentzian case.

\subsection*{The pseudo-Euclidian space}

Let~$Q = x_1^2 + \cdots + x_p^2 - y_1^2-\cdots-y_q^2$ be the standard quadratic form of signature~$(p,q)$ on~$\R^{p+q}$. The pseudo-Riemannian metric defined on~$\R^{p+q}$ which is constant equal to~$Q$ is called the \emph{pseudo-Euclidian metric} on~$\R^{p+q}$. The space~$(\R^{p+q}, Q)$ is called the \emph{pseudo-Euclidian space} and is usually denoted by~$\R^{p,q}$. The pseudo-Euclidian space is the model space for pseudo-Riemannian manifolds of curvature~$0$.

\subsection*{Pseudo-Riemannian products}

Let~$(P, g_P)$ and~$(N, g_N)$ be two Riemannian manifolds. One may define the pseudo-Riemannian product of~$P$ and~$N$ as~$(P \times N, g_P - g_N)$. 

\begin{lemma}\label{lemmacausalgeodesic}
When $P$ and $N$ are complete, if two points $(x_1, y_1)$ and $(x_2, y_2)$ in $P \times N$ verify $d_P(x_1, x_2) \leqslant d_N(y_1, y_2)$ (resp. $d_P(x_1, x_2) < d_N(y_1, y_2)$), then there is a causal (resp. timelike) geodesic between them.
\end{lemma}

\begin{proof}
Since $P$ and $N$ are complete, one may consider the geodesics $\gamma_P$ and $\gamma_N$ parametrized by $[0,1]$, linking $x_1$ to $x_2$ and $y_1$ to $y_2$ having lengths $d_P(x_1, x_2)$ and $d_N(y_1, y_2)$. The path $(\gamma_P, \gamma_N) : [0,1] \rightarrow P \times N$ is then a causal (resp. timelike) geodesic between $(x_1, y_1)$ and $(x_2, y_2)$, hence the result.
\end{proof}

\subsection*{The Einstein space}

Let~$Q$ be the standard quadratic form of signature~$(p+1,q+1)$ on~$\R^{p+1,q+1}$ and let~$\widehat{\ein}^{p,q}$ be the space of isotropic half-lines in~$\R^{p+1,q+1}$.

\begin{prop}
The metric~$Q$ endows~$\widehat{\ein}^{p,q}$ with a conformal pseudo-Riemannian metric.
\end{prop}

\noindent The Einstein space is conformally flat, meaning that any point in~$\widehat{\ein}^{p,q}$ admits a neighborhood which is conformally equivalent to an open set of the pseudo-Euclidian space. In fact, the Einstein space admits charts which are conformally equivalent to the whole of~$\R^{p,q}$.

\begin{defn}
Let~$x \in \widehat{\ein}^{p,q}$ and let~$\C(x)$ be the projectivisation in~$\widehat{\ein}^{p,q}$ of~$x^{\perp} \cap \C$.~$\C(x)$ is called the \emph{isotropic cone} of~$x$. It is also the reunion of all the lightlike geodesics containing~$x$.
\end{defn}

\begin{prop}
    Let~$x \in \widehat{\ein}^{p,q}$ and let~$\Omega(x) = \widehat{\ein}^{p,q} \setminus \C(x)$. The set~$\Omega(x)$ is the union of two connected components, each of which is conformally equivalent to the pseudo-Euclidian space~$\R^{p,q}$. When considering the quotient by antipody~$\ein^{p,q}$,~$\Omega(x)$ is connected and conformally equivalent to~$\R^{p,q}$.
    \end{prop}

\noindent The Einstein space also admits a nice conformal model.

\begin{prop}
The Einstein space is conformally equivalent to the split product~$(\s^p \times \s^q, g_{\s^p} - g_{\s^q})$ where~$g_{\s^p}$ and~$g_{\s^q}$ are the round metrics on~$\s^p$ and~$\s^q$.
\end{prop}

\subsection*{The pseudo-spherical space}

Let~$Q$ be the standard quadratic form of signature~$(p+1,q)$ on~$\R^{p+1,q}$ and let~$\s^{p,q}$ be the hypersurface~$(Q = 1)$ in~$\R^{p+1,q}$.

\begin{prop}
    The metric~$Q$ in~$\R^{p+1,q}$ endows by restriction~$\s^{p,q}$ with a Lorentzian metric of signature~$(p,q)$.
\end{prop}

The space $\s^{p,q}$ has constant curvature $1$. It is simply connected when $p \geqslant 2$.

\subsection*{The pseudo-hyperbolic space}

Let~$Q$ be the standard quadratic form of signature~$(p,q+1)$ on~$\R^{p,q+1}$ and let~$\h^{p,q}$ be the hypersurface~$(Q = -1)$ in~$\R^{p,q+1}$.

\begin{prop}
The metric~$Q$ in~$\R^{p,q+1}$ endows~$\h^{p,q}$ with a pseudo-Riemannian metric of signature~$(p,q)$.
\end{prop}

\noindent When $q \geqslant 2$, the pseudo-hyperbolic space is the model for pseudo-Riemannian manifolds of curvature~$-1$.

\subsection{Causality in the pseudo-Euclidian space}\label{causalityminkowski}

In this section, we will assume that~$M$ is the pseudo-Euclidian space 
$\R^{p,q}$, i.e the flat space~$\R^{p+q}$ endowed with the constant pseudo-Riemannian 
metric~$g = dx_1^2 + \cdots + dx_p^2 - dy_1^2 - \cdots - dy_q^2$. When~$q = 1$,~$M$ 
is the Minkoswki space~$\R^{p,1}$ and is the simplest Lorentzian space in which we can study causality. Since we have already generalized the notion of time orientation, 
this gives us a good motivation for the elaboration of a non-trivial notion of causality in~$\R^{p,q}$. 

A first natural attempt would be to say that a vector~$v$ is causally related to the origin 
if and only if there exists a causal path~$c \colon I \rightarrow \R^{p,q}$ which starts at the origin 
and ends at the vector~$v$, however, it is easy to see that this definition is trivial:

\begin{prop}
    Assume $q \geqslant 2$ and let~$v \in \R^{p,q}$. There exists a causal path~$c \colon I \rightarrow \R^{p,q}$ 
    which goes from the origin to~$v$.
\end{prop}

\begin{figure}[h!bt]
    \begin{center}
      \includegraphics[width=.7\linewidth]{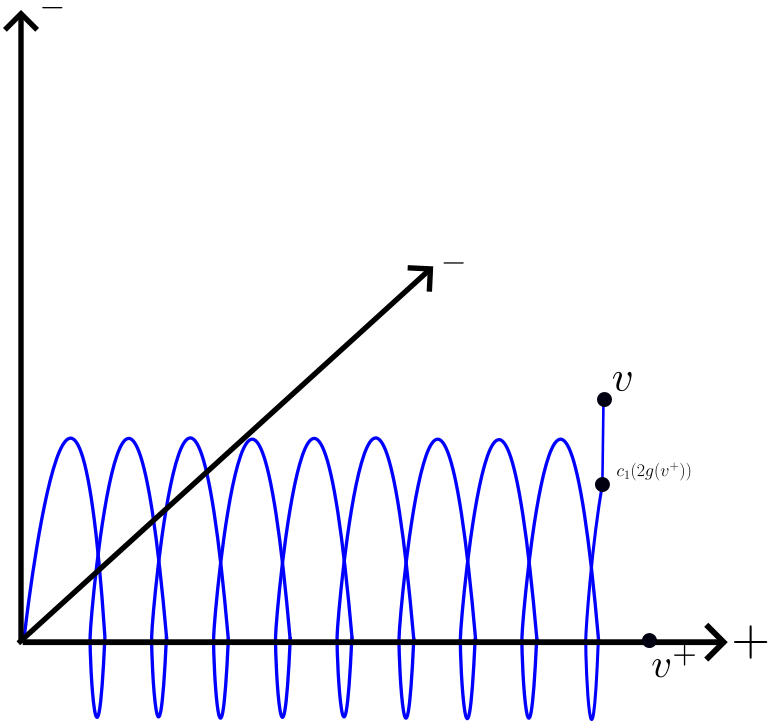}
      \end{center}
      \caption[]{The timelike curve goes from the origin to the vector~$v$.}
    \end{figure}\label{causalpathimage}

\begin{proof}
    Let~$v^+ = \pi_{\R^{p,0}}(v)$. If $v^+ = 0$, $v$ is in $\R^{0,q}$ and a simple segment does the trick. Otherwise, let us define the path 

    \[c_1 \colon t \in [0, 2 g(v^+)] \longmapsto t \frac{v^+}{2g(v^+)} + (\cos(t)-1) e_{q+1} + \sin(t) e_{q+2}, \]

    \noindent for which~$\forall t < 2 g(v^+)$,~$g(\dot{c_1}(t)) = - \frac{1}{2}$,~$c_1(0) = 0$ and~$\pi_{\R^{p,0}}(c_2(2 g(v^+))) = v^+$.
    Since~$\pi_{\R^{p,0}}(c_1(2 g(v^+))) = v^+$, it is possible to take a smooth path~$c_2$ which goes from~$c_1(2 g(v^+))$ to~$v$ while staying in 
    an affine copy of~$\R^{0,q}$. Up to a smoothing, the concatenation of~$c_1$ and~$c_2$ gives a smooth timelike path from the origin to $v$ (see picture \ref{causalpathimage}).
\end{proof}

\noindent This naïve approach to causality in higher signature seems to be doomed. Another natural 
generalization would be to consider smooth timelike immersions of dimension~$q$ instead of smooth paths, 
since the notion of time orientation we have introduced concerns timelike~$q$-vector spaces. 
One would then be tempted to define that~$v$ is causally related to the origin if and only if there exists 
a timelike immersion~$f \colon N_0 \rightarrow \R^{p,q}$ of dimension~$q$ containing in its image both~$v$ and the origin. 
However, this definition also fails to yield an interesting relation.

\begin{prop}
    Let~$v \in \R^{p,q}, q \geqslant 2$. There exists a timelike immersion~$f \colon N_0 \rightarrow \R^{p,q}$ 
    of dimension~$q$ containing both~$v$ and the origin in its image.
\end{prop}

\begin{proof}
    Let~$c \colon I \rightarrow \R^{p,q}$ be a smooth timelike path from a segment~$I$ to~$\R^{p,q}$ 
    containing both~$v$ and the origin in its interior. This path can be constructed in the same manner as before. Since~$I$ is compact, 
    it is possible to consider a tubular neighborhood~$\overline{c} \colon \Omega \rightarrow \R^{p,q}$, 
    with~$\Omega = \mathring{I} \times D^{p}(0, r) \times D^{q-1}(0,r)$,~$\overline{c}|_{\mathring{I}} = c|{\mathring{I}}$ and~$\overline{c}$ an immersion. 
    The space~$\Omega$ can be endowed with a pseudo-Riemannian metric of signature~$(p,q)$ via the pullback of~$g$ by~$\overline{c}$. 
    We now only have to take~$N_0 = \mathring{I} \times D^{q-1}(0,r)$. We have thickened~$c$ into a timelike immersion~$\overline{c}$ of dimension~$q$ whose image 
    contains the image of~$c$, and thus both the vector~$v$ and the origin (see picture \ref{causalpathtick}).
\end{proof}

\begin{figure}[hbt]
    \begin{center}
      \includegraphics[width=.7\linewidth]{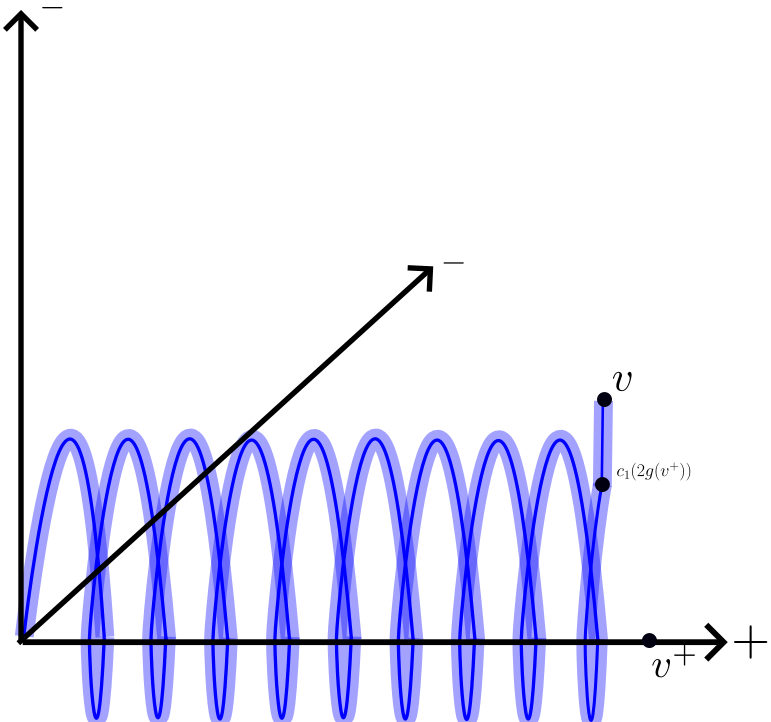}
      \end{center}
      \caption[]{The same curve, thickened into a timelike immersion of dimension~$2$.}
    \end{figure}\label{causalpathtick}

\noindent The property we need in order to get an interesting notion of causality 
in~$\R^{p,q}$ is that of \emph{inextendability} for causal immersions of dimension $q$. We want our notion of inextendability to generalize the one for paths. Intuitively, we want for the immersion~$f$ to be "without boundary" within~$M$. Being proper does not satisfy our needs as we would like the path~$\R \rightarrow \s^1$ to be inextendible. Only taking the boundary of the image of~$f$ also does not work ; One would like for the endpoints to be accounted for in the boundary of~$f$. 

\begin{comment}

\begin{defn}
    Let~$M$ be a~$(p,q)$-spacetime. A continuous map~$f \colon N_0 \rightarrow M$ from a~$k$-dimensional smooth manifold~$N_0$ 
    is said to be \emph{inextendible} if for each continuous path~$c \colon I \rightarrow M$ which 
    is inextendible in~$N_0$,~$f \circ c$ is inextendible in~$M$.
\end{defn}

\end{comment}

\begin{defn}
    Let~$M$ be a~$(p,q)$-spacetime. A continuous map~$f \colon N_0 \rightarrow M$ from a~$k$-dimensional manifold~$N_0$ 
    is said to be \emph{inextendible} if for each continuous path~$c \colon I \rightarrow M$ which 
    is inextendible in~$N_0$,~$f \circ c$ is inextendible in~$M$.
\end{defn}

\noindent Note that any path which is not inextendible can always be extended into an inextendible 
path. That is not the case for immersions of higher dimensional manifolds.

\begin{defn}
    Let~$x \in \R^{p,q}$. We will denote by~$I(x)$ (resp.~$J(x)$) the set 
    of points~$y \in \R^{p,q}$ for which there exists an inextendible timelike (resp. causal) immersion 
    of dimension~$q$ containing both~$x$ and~$y$. When~$y \in I(x)$ (resp.~$y \in J(x)$), the set~$\{x,y\}$ will be said 
    to be in \emph{timelike position} (resp. \emph{causal position}). More generally, when a set~$E \subset \R^{p,q}$ is 
    included in the image of an inextendible timelike (resp. causal) immersion of dimension~$q$, the set~$E$ will 
    be said to be in timelike position (resp. causal position), or simply that~$E$ is timelike (resp. causal).
\end{defn}

\noindent In order to show that this notion of causality in non-trivial, let us study the inextendible timelike immersions of dimension~$q$ in~$\R^{p,q}$.
Let~$T \colon \R^{p,q} \rightarrow \R^q$ be the projection on the negative coordinates of the pseudo-Euclidian space.

\begin{prop}\label{EuclidianGH}
    Let~$f \colon N_0 \rightarrow \R^{p,q}$ be an inextendible causal immersion of dimension $q$. Then 
    the map~$T \circ f \colon N_0 \rightarrow \R^q$ is a diffeomorphism. In particular,~$f$ can be seen as the graph of a~$1$-Lipschitz map 
    from~$\R^q$ to~$\R^p$.
\end{prop}

\begin{figure}[h!bt]
    \begin{center}
      \includegraphics[width=.7\linewidth]{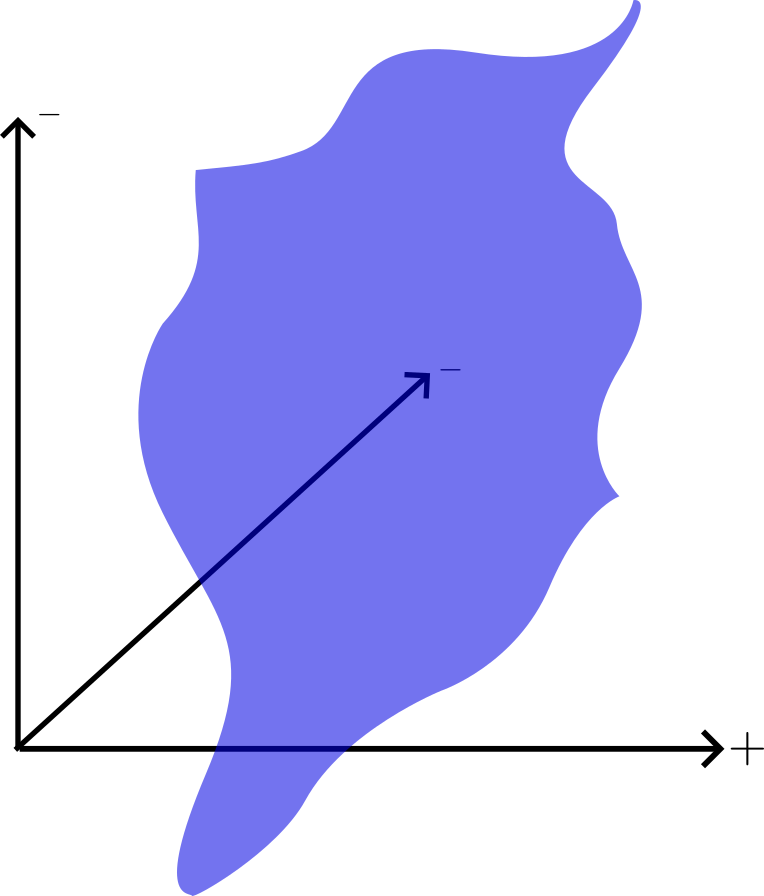}
      \end{center}
      \caption[]{An inextendible timelike immersion of dimension~$2$ in~$\R^{1,2}$.}
    \end{figure}

\begin{proof}
    This proof adapts an argument by Mess in \cite{mess2007lorentz}. Let us first show that the map~$T \circ f$ is a local diffeomorphism. Since the projection~$T$ removes the negative coordinates, the map~$T \circ f$ increases the distances between the manifolds~$(N_0, -f^* g_{\R^{p,q}})$ and~$(\R^q, g_{\R^q})$, i.e for each~$x \in N_0, v \in T_x N_0$, 

    \[ g_{\R^q}(d(T \circ f)(v)) \geqslant - f^* g_{\R^{p,q}}(v). \] 

    \noindent This tells us in particular that~$T \circ f$ is a local diffeomorphism ; in particular, its image~$\Imm(T \circ f)$ is open in~$\R^q$.

    Let~$c \colon [a,b] \rightarrow \R^q$ be a smooth path of speed~$1$ in $\R^q$ with $c(a)$ in $\Imm(T \circ f)$ and let us assume that $c$ may only be pulled back to a map $\tilde{c} : [a,c) \subset [a,b] \rightarrow N_0$ on $N_0$ via $T \circ f$. We will denote by~$\overline{g} = g_{\R^p} + g_{\R^q}$ the flat Riemannian metric on~$\R^{p,q} = \R^{p+q}$. For each~$x \in \R^{p,q}$ and~$v \in T_x \R^{p,q}$ timelike, we have that 

    \[\begin{split}
        \overline{g}(v) &= g_{\R^q}(v) + g_{\R^q}(v) \\ 
        &\leqslant 2 g_{\R^q}(v), 
    \end{split} \] 

    \noindent since~$g_{\R^{p,q}}(v) \leqslant 0$. In particular, since $T \circ f \circ \tilde{c} = c$, we have that the speed of~$f(\tilde{c})$ for~$\overline{g}$ is lower than~$2$. Since~$\R^{p,q}$ is complete,~$f(\tilde{c})$ must converge, and so does $\tilde{c}$ since~$f$ is inextendible. However we have assumed that $\tilde{c}$ could not be extended further, hence the contradiction. This proves that any smooth path of $\R^q$ starting in $\Imm(T \circ f)$ may be completely pulled back into a path of $N_0$, thus $T \circ f$ is surjective and is a covering of $\R^q$.

     Finally, since~$\R^q$ is simply connected, the map~$T \circ f$ must be a diffeomorphism from~$N_0$ to~$\R^q$, hence the result. In particular, the image of $f$ is the graph of a map from $\R^q$ to $\R^p$, and one may easily see that such graphs are causal if and only if the map is $1$-Lipschitz, hence the result.
\end{proof}

\begin{comment}

\begin{prop}
    Let~$f \colon N_0 \rightarrow \R^{p,q}$ be an inextendible causal immersion. Then 
    the map~$T \circ f \colon N_0 \rightarrow \R^q$ is a diffeomorphism. In particular,~$f$ can be seen as the graph of a~$1$-Lipschitz map from~$\R^q$ to~$\R^p$ which is~$1$-contracting if and only if~$f$ is timelike.
\end{prop}

\begin{proof}
    Let~$g_-$ be the constant~$(p,q)$ metric on~$\R^{p+q}$ defined by

    \[g_- = dx_1^2 + \cdots dx_p^2 - 2dy_1^2 - \cdots - 2dy_q^2.\] 

    \noindent It is clear that~$g_- < g$. In particular, every causal immersion for~$g$ is a timelike immersion for~$g_-$. Thus~$f \colon N_0 \rightarrow \R^{p+q}$ is an inextendible timelike immersion for~$(\R^{p+q}, g_-)$. With the same reasoning as before, we can show that the projection~$T \colon \R^{p,q} \rightarrow \R^q$ verifies that~$T \circ f$ is a diffeomorphism between~$N_0$ and~$\R^q$, hence the first part of the result. The map~$f$ can then be seen as the graph of a smooth map from~$\R^q$ to~$\R^p$, and it is clear that this map is~$1$-Lipschitz if and only~$f$ is causal and~$1$-contracting if and only if~$f$ is timelike.
\end{proof}

\end{comment}

\noindent In particular, this makes it easy to realize that the causality relation between two points 
coincides with the nature of the segment joining them.

\begin{prop}\label{flatgeodesics}
    Let~$x,y$ be two points in~$\R^{p,q}$. The set~$\{x,y\}$ is in timelike position (resp. causal position) if and only if the segment~$[x,y]$ is timelike (resp. causal).
\end{prop}

\begin{proof}
    Assume that~$[x,y]$ is timelike (resp. causal) and include the vector~$y-x$ into a vector space~$V$ of dimension~$q$ such that~$g|_V < 0$ (resp.~$g|_V \leqslant 0$). 
    The space~$x + V$ is an inextendible timelike (resp. causal) immersion.

    \indent Conversely, assume that there exists an inextendible causal immersion~$f$ of dimension $q$ containing both~$x$ and~$y$. This immersion can be seen as the graph of a~$1$-Lipschitz map~$h$ from~$\R^q$ to~$\R^p$. In particular, if~$x = (h(x_0), x_0)$ and~$y = (h(y_0), y_0)$, we have that
    
    \[ d(h(x_0), h(y_0)) \leqslant d(x_0, y_0),\]

    \noindent thus~$[x,y]$ is a causal segment (see lemma \ref{lemmacausalgeodesic}). When the immersion is timelike, the map~$h$ is~$1$-contracting and~$[x,y]$ is timelike, hence the result.
\end{proof}

\noindent We can actually characterise the subsets of~$\R^{p,q}$ which are in causal position.

\begin{prop}\label{timeposition}
    Let~$E$ be a subset of~$\R^{p,q}$. The set~$E$ is in causal position if and only if 
    for each~$x,y \in E$, the segment~$[x,y]$ is in causal position.
\end{prop}

\begin{proof}
    Assume that~$E$ is in causal position. Then for each~$x,y$ in~$E$, the set~$\{x,y\}$ is in causal position, meaning that the segment~$[x,y]$ is causal according to proposition \ref{flatgeodesics}.

    Conversely, if~$E$ verifies that for each~$x,y \in E$,~$[x,y]$ is causal, then~$E$ is the graph 
    of a~$1$-Lipschitz map from a subset~$T(E)$ of~$\R^q$ to~$\R^p$. This map can then be extended into a 
   ~$1$-Lipschitz map defined on the whole of~$\R^q$ by Kirszbraun's Theorem (see \cite{schwartz1969nonlinear}, p. 21). Then, by definition,~$E$ is in causal position.
\end{proof}

\begin{rem}\label{lipschitzextension}
    We may wish to also characterise the sets~$E$ of~$\R^{p,q}$ which are in timelike position; such a set must be the graph of a ~$1$-contracting map~$g \colon T(E) \rightarrow \R^p$, i.e such that~$d(g(x), g(y)) < d(x,y)$. However Kirszbraun's Theorem does not guarantee the existence of a ~$1$-contracting extension. The proof only works if~$E$ is the graph of a~$k$-Lipschitz map with~$k < 1$.
\end{rem}

\noindent We will come back for a more in-depth study of the pseudo-Euclidian space 
later on.

\subsection{Local structure of causal immersions}

We can extend the notion of causality we introduced on~$\R^{p,q}$ to any~$(p,q)$-spacetime 
with the same definitions.

\begin{defn}\label{causalset}
    A set~$E$ of~$M$ is said to be in timelike position (resp. causal position) when there 
    exists an inextendible timelike (resp. causal) immersion of dimension $q$ containing~$E$ in its image.
\end{defn}

\noindent This notion only depends on the conformal class of the metric. A first fact we can observe is that depending on the spacetime~$M$, it is possible 
that there does not exist any timelike or causal immersion which is inextendible in the sense 
we defined as soon as~$q \geqslant 2$. We give examples of such spacetimes in \ref{mainexample}.

\begin{defn}
    A~$(p,q)$-spacetime~$M$ is said to have non-trivial causality if for each point~$x$ in~$M$, the set~$\{x\}$ is in timelike position, meaning there always exists an inextendible timelike immersion of dimension~$q$ containing~$x$.
\end{defn}

\noindent Please note that in the Lorentzian case, every spacetime 
has non-trivial causality. In Lorentzian geometry, it is a well known fact that the local causality in a normal neighborhood of a point~$x$ 
is entirely determined by the sign of the geodesics originating from~$x$. We are going to show a similar result in higher signatures.

\begin{defn}
Let $g_1$, $g_2$ be two pseudo-Riemannian metrics on a manifold $M$. We will say that $g_1 \leqslant g_2$ (resp. $g_1 < g_2$) if for each $x \in M$ and each $v \in T_x M$, if $v$ is causal for $g_2$, then it is causal (resp. timelike) for $g_1$.
\end{defn}

\begin{lemma}
    Let~$x \in M$. There exist an exponential neighborhood~$\exp : U \subset T_x M \rightarrow M$ centered on $x$ and two pseudo-Riemannian metrics~$g_-$ and~$g_+$ constant on~$U$ such that~$g_- < \exp^* g < g_+$. 
\end{lemma}

\begin{proof}
    Let $\exp : V \subset T_x M \rightarrow M$ be an exponential neighborhood centered on $c$. Let us write $(\exp^* g)|_0 = g_x = x_1^2 + \cdots + x_p^2 - y_1^2 - \cdots - y_q^2$ in a good basis of $T_x M$ and let~$g_-$ and~$g_+$ be the constant metrics on~$V$ defined by

    \[g_- = x_1^2 + \cdots + x_p^2 - (1+\varepsilon)y_1^2 - \cdots - (1+\varepsilon)y_q^2,\]

    \[g_+ = x_1^2 + \cdots + x_p^2 - (1-\varepsilon) y_1^2 - \cdots - (1-\varepsilon) y_q^2,\]

    \noindent with $\varepsilon > 0$. Since~$\exp^* g$ is smooth on~$V$, we know that there exists a neighborhood~$U$ of~$0$ in~$V$ such that in~$U$, we have 

    \[ g_- < \exp^* g < g_+,\]

    \noindent hence the result. Taking $\varepsilon$ to be near zero gives finer approximations of $\exp^* g$ but reduces the size of the neighborhood $U$.
\end{proof}

\begin{defn}
    We will take~$U$ to be of the shape~$C(a,b) = (-a,a)^p \times \D^q(0, b)$ with~$a \gg b$ after chosing an orthonormal basis of $T_x M$. We will call such a neighborhood a \emph{cylindrical neighborhood}.
\end{defn}

\begin{lemma}\label{lemmalorentzian}
    Let~$U = C(a,b) \subset T_x M$ be a cylindrical neighborhood of~$x$ with constant metrics~$g_-$ and~$g_+$ and let~$v \in U$ such that~$g_-(v) < 0$. Let~$F$ be a~$p$-dimensional vector space in~$U$ which is positive for~$g_-$ and orthogonal to~$v$ for~$g_-$. Let~$L = F \oplus v$. The space~$L \cap U$ is a Lorentzian manifold in~$(U,\exp^* g)$.
\end{lemma}

\begin{proof}
    Since~$g_-(v) < 0$,~$g_-$ is of signature~$(p,1)$ on~$L$.
    Furthermore~$\dim(L) = p+1$, thus~$g_+$ cannot be definite positive on~$L$, meaning that there exists a non-zero~$w$ such that~$g_+(w) \leqslant 0$. Let~$y \in L$. Since~$g_- < g$,~$g_L = g|_L$ must be 
    positive on~$F \subset T_y L$, meaning that~$g_L|_y$ is either of signature~$(p,1)$ or degenerate. However, since~$g < g_+$ we must have~$g_L(w) < 0$, and~$g_L$ must be of signature~$(p,1)$.
\end{proof}

\begin{prop}\label{geodesics}
    Let~$x$ be a point in a spacetime~$M$. There exists a normal neighborhood~$U \subset T_x M$ such that if there exists a timelike immersion of dimension $q$~$f \colon N_0 \rightarrow U$ inextendible in~$U$ containing both~$0$ and~$v \in V$, then the segment~$[0, v]$ in $(T_x M, g_x)$ is timelike.
\end{prop}

\begin{proof}
    Let~$U = C(a,b)$ be a cylindrical neighborhood around~$x$ defined as before and let~$f \colon N_0 \rightarrow U$ be a timelike immersion of dimension $q$ inextendible in~$U$. 
    Since~$f$ is timelike for~$g$, it must also be timelike for~$g_-$. We then know with proposition \ref{flatgeodesics} that the composition of~$f$ with the projection onto~$\D^q(0,b)$ must be a diffeomorphism.
    In other words,~$f$ can also be regarded as a map~$\overline{f} \colon \D^q(0,b) \rightarrow (-a,a)^p$ whose graph is the image of the inextendible immersion~$f$. 

    \begin{figure}[hbt]
        \begin{center}
          \includegraphics[width=.9\linewidth]{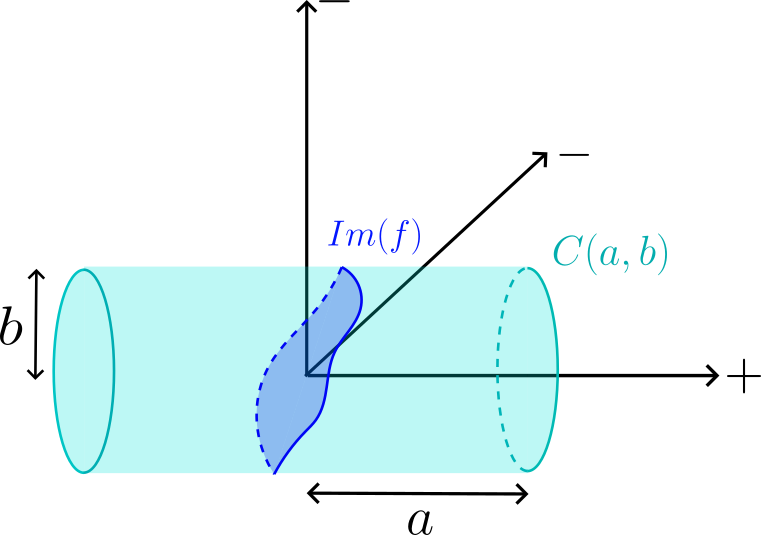}
          \end{center}
          \caption[]{An inextendible timelike immersion~$f$ inside a cylindrical neighborhood~$C(a,b)$.}
        \end{figure}

        \noindent Let~$v$ be in the image of~$f$. We are going to show that the line~$[0, v]$ is timelike. 
    What we have done before shows that~$g_-(v) < 0$. Let~$F$ be a~$p$-dimensional vector space in~$U$ which is positive for~$g_-$ (and thus for~$g_+$) and orthogonal to~$v$ 
    for~$g_-$. Let~$L = F \oplus v$. By lemma \ref{lemmalorentzian}, we know that~$L$ is a Lorentzian manifold for~$g$. 

    Since $L$ is has signature $(p,1)$ and the image of $f$ has signature $(0,q)$, the two manifolds are always transverse at their intersection which must then a one dimensional manifold. As both $L$ and the image of $f$ are closed, the intersection must also be closed, hence it is a connected path. Since it is contained in~$\Imm(f)$, this path must be timelike. However, since it is also contained 
    in the Lorentzian space~$L$ and since~$L \cap U$ is a normal neighborhood for~$L$, results of Lorentzian geometry tell us that the geodesic from the origin 
    to~$v$ in~$L$ must be timelike. Since~$L$ is a vector space in the normal neighborhood~$U$, the geodesic~$[0,v]$ in~$M$ is contained in~$L$, meaning that~$[0,v]$ must also be the geodesic in~$L$, hence the result.
\end{proof}

\noindent We can prove the same result when the immersion is causal.

\begin{prop}
     Let~$x$ be a point in~$M$. There exists a normal neighborhood~$U \subset T_x M$ such that if there exists an inextendible causal immersion~$f \colon N_0 \rightarrow U$ of dimension $q$ containing both~$0$ and~$v \in U$, then~$[0, v]$ is causal. Furthermore, when~$[0,v]$ is lightlike, the image of the immersion must contain the whole segment~$[0,v]$.
\end{prop}

\begin{proof}
    We proceed in the same fashion as previously. Let $D(a,b)$ be a cylindrical neighborhood around $x$ and $g_-$ a constant metric verifying $g_- < g$. Since~$f$ is causal for~$g$, it must be 
    timelike for~$g_-$, meaning that it must be the graph of a map~$\overline{f} \colon \D^q(0,b) \rightarrow (-a, a)^p$. 
    Let~$F$ be a~$p$-dimensional vector space in~$U$ which is positive for~$g_-$ and orthogonal to~$v$ for~$g_-$ and let~$L = F \oplus v$. By lemma \ref{lemmalorentzian},~$L$~must be Lorentzian for~$g$. The intersection~$\Imm(f) \cap L$ is a path between the origin and~$v$ 
    which is causal in~$L$. This means that the line~$[0,v]$ is either timelike or lightlike. In the case where it is lightlike, results of Lorentzian geometry (see \cite{largescalestructure}, proposition 4.5.1) tell us that the path~$\Imm(f) \cap L$ must be the lightlike geodesic~$[0,v]$, hence the result.
\end{proof}

\noindent It also gives us a characterization of smooth maps which are causal immersions, meaning that 
$f \colon N_0 \rightarrow M$ is a timelike (resp. causal) immersion if and only if for each~$y \in N_0$, there exists a cylindrical neighborhood 
$U = C(a,b)$ centered at~$x = f(y)$ such that the restriction of~$f$ to~$f^{-1}(U)$ 
is the graph of a map from~$\D^q(0,b)$ to~$(-a, a)^p$ which takes values in the area~$\{g|_0 < 0\} \subset U$ (resp.~$\{g_x \leqslant 0\}$). 
This characterisation is interesting, as it can easily be generalized when~$f$ is only a continuous map, which gives us a bigger class of maps that 
we can call causal.

\begin{defn}\label{continuouscausal}
    A continuous map~$f \colon N_0 \rightarrow M$ from a smooth manifold~$N_0$ to~$M$ will be said to be timelike (resp. causal) if for each~$y \in N_0$, there exists a cylindrical neighborhood~$U = C(a,b)$ of~$x = f(y)$ such that the restriction 
    of~$f$ to the connected component of~$f^{-1}(U)$ containing $x$ is the graph of a map from~$\D^q(0,b)$ to~$(-a, a)^p$ taking values in the area~$\{g|_x < 0\}$ (resp.~$\{g|_x \leqslant 0\}$).
\end{defn}

\noindent This generalization will become useful later on when we will construct limit maps which are not necessarily smooth. 

\begin{prop}\label{diffalmosteverywhere}
    Let~$f \colon N_0 \rightarrow M$ be a continuous causal map. Then~$f$ is differentiable almost everywhere and for each~$x$ in which~$f$ is differentiable,~$\Imm(df_x)$ is causal.
\end{prop}

\begin{proof}
    Let~$x \in N_0$ and~$C(a,b)$ a cylindrical neighborhood around~$f(x)$ such that the image of~$f$ in~$C(a,b)$ is the graph of a map~$\hat{f} \colon \D^q(0,b) \rightarrow (-a,a)^p$ with image in~$\{g_x \leqslant 0\}$. Since~$g_- < g$, the map~$f$ is also timelike for~$g_-$ as in each point, the lightcone of~$g_-$ contains the lightcone of~$g$. Thus the map~$\hat{f}$ must be~$1$-Lipschitz for~$g_-$, which guarantees that~$\hat{f}$ is differentiable almost everywhere in~$\D^q(0,b)$ by Rademacher's theorem (\cite{Rademacher}). Assume now that~$f$ is differentiable at~$x$ and that~$\Imm(df_x)$ is not causal. There must exist a direction on which~$df$ is positive, meaning that~$f$ cannot take values in~$\{g|_0 \leqslant 0\}$, hence the result.
\end{proof}

    \begin{prop}
    Let~$d$ be a Riemannian distance on~$M$ and~$f \colon N_0 \rightarrow M$ a continuous timelike map. For each positive continuous map~$\delta \colon N_0 \rightarrow \R$, there exists a smooth timelike immersion~$g \colon N_0 \rightarrow M$ such that for each~$x \in N_0$, 

    \[d(f(x), g(x)) \leqslant \delta(x).\]
\end{prop}

\begin{proof}
    Let us consider a locally finite covering of compacts~$(U_i)$ of~$N_0$. Let~$x \in U_i$. Since~$f$ is timelike, one can find a cylindrical neighborhood~$C(a,b)$ of~$f(x)$ such that~$f$ can be seen as the graph of a map~$\hat{f} \colon \D^q(0,b) \rightarrow (-a,a)^p$. For each~$\varepsilon >0$, let us define a metric
    
    \[g_{\varepsilon} = x_1^2+\cdots+x_p^2 - (1 - \varepsilon)(y_1^2 + \cdots + y_q^2)\]

    \noindent which is constant on~$C(a,b)$ and which verifies~$g_{\varepsilon}|_0 > g|_0$. Since~$f$ is timelike and by taking a small enough neighborhood~$V_x = C(a',b') \subset C(a,b)$, one can find a~$\varepsilon$ such that~$f$ is timelike for~$g_{\varepsilon}$ and~$g_{\varepsilon} > g$ on~$C(a', b')$. In somes cases, this may require a continuous deformation of~$\hat{f}$ near zero, for instance when~$f$ is a timelike path in a lorentzien manifold for which~$g(\dot{f}(x)) = 0$. 

    \indent The map~$f$ is timelike for~$g_{\varepsilon}$, meaning that~$\hat{f}$ must be a~$1$-contracting map for~$g_{\varepsilon}$ from~$\D^q(0,b)$ to~$(-a,a)^p$. Since~$g_{\varepsilon} > g$, it is enough to find a smooth approximation of~$\hat{f}$ which is~$1$-contracting for~$g_{\varepsilon}$ ; one may find such a uniform approximation by taking a convolution of~$\hat{f}$. Furthermore, it is possible to ask that this approximation be equal to~$\hat{f}$ on~$\D^q(0,b)$.
    
    \indent Since the covering~$(U_i)$ is locally finite and each~$\overline{U_i}$ is compact, it is possible to extract a locally finite covering of the~$(V_x)$. We can then build the smooth timelike immersion~$g$ on the entirety of~$N_0$, hence the result.
\end{proof}

\begin{rem}
    With this same proof, it is also possible to approximate~$f$ with a timelike immersion~$g$ such that~$g = f$ on a finite set of points.
\end{rem}

\noindent The proof for the approximation of causal, non necessarily timelike curves is slightly more involved.

\begin{prop}\label{approxcausal}
    Let~$d$ be a Riemannian distance on~$M$ and~$f \colon N_0 \rightarrow M$ a continuous causal map. For each positive continuous map~$\delta \colon N_0 \rightarrow \R$, there exists a smooth causal immersion~$g \colon N_0 \rightarrow M$ such that for each~$x \in N_0$, 

    \[d(f(x), g(x)) \leqslant \delta(x).\]
\end{prop}

\begin{proof}
    Let~$(U_i)$,~$x \in U_i$,~$C(a,b)$ be as before. By proposition \ref{diffalmosteverywhere}, we know that~$f$ is differentiable almost everywhere. Let~$x \in U_i$ such that~$f$ is differentiable in~$x$. The sub-vector space~$\Imm(df_x)$ in~$C(a,b)$ must be causal, meaning it must be the orthogonal sum of an isotropic vector space~$F$ of dimension~$k$ and a negative vector space~$Q$ of dimension~$q-k$. Let us decompose~$T_{f(x)} M = \R^{k,k} \oplus \R^{p-k} \oplus Q$ where~$F \subset \R^{k,k}$. Finally, let us define the constant metric on~$C(a,b)$

    \[g_{\varepsilon} = g|_{\R^{k,k}} + g|_{\R^{p-k}} + (1- \varepsilon)g|_Q \]

    \noindent for each~$\varepsilon >0$. For a small enough neighborhood~$V_x = C(a', b')$ and a small enough epsilon, we have that~$g_{\varepsilon} \geqslant g$ and that~$f$ is causal for~$g_{\varepsilon}$. The map~$\hat{f}$ must then be~$1$-Lipschitz for~$g_{\varepsilon}$ and we can uniformaly approximate~$\hat{f}$ by a smooth~$1$-Lipschitz map for~$g_{\varepsilon}$. Since~$g_{\varepsilon} \geqslant g$, this map is still causal for~$g$, hence the result.
\end{proof}

\noindent Since most of our results concern smooth causal immersions which are inextendible, let's show an even stronger result when the continuous map~$f$ is assumed to be inextendible.

\begin{lemma}
    Let~$f \colon N_0 \rightarrow M$ be an inextendible continuous causal (resp. timelike) map from a smooth manifold~$N_0$ to~$M$ and let~$\delta \colon N_0 \rightarrow \R$ be a positive continuous map. There exists a smooth inextendible causal (resp. timelike) immersion~$g \colon N_0 \rightarrow M$ such that for each~$x \in N_0$,~$d(f(x), g(x)) < \delta(x)$.
\end{lemma}

\begin{proof}
    Let~$\varphi$ be a proper continuous map greater than~$1$ on~$N_0$ and let~$\overline{\delta} = \frac{1}{\varphi} \delta$. The map~$\overline{\delta}$ verifies that~$\overline{\delta} < \delta$ and that for each~$\varepsilon > 0$, there exists a compact~$K$ of~$N_0$ such that~$\overline{\delta} < \varepsilon$ outside of~$K$. Proposition \ref{approxcausal} tells us that there exists a causal (resp. timelike) immersion~$g \colon N_0 \rightarrow M$ which verifies that~$d(f,g) \leqslant \overline{\delta} \leqslant \delta$. Let us show that~$g$ must be inextendible.

    \indent Let~$c$ be an inextendible path in~$N_0$ and assume that~$g \circ c$ converges in~$M$ to a point~$\ell$. Let~$\varepsilon > 0$. Let~$T_1 \in \R$ such that for all~$t > T_1$,~$d(g(c(t)), \ell) < \varepsilon$. Let~$K$ be a compact subset of~$N_0$ such that~$\overline{\delta} < \varepsilon$ outside of~$K$. Since~$c$ is inextendible in~$N_0$, it exits every compact of~$N_0$ ; in particular, there exists a~$T_2 > T_1$ such that for all~$t > T_2$,~$c(t) \notin K$. We then have for all~$t > T_2$,

    \[ \begin{split}
    d(f(c(t)), \ell) &\leqslant d(f(c(t)), g(c(t))) + d(g(c(t)), \ell) \\
    &\leqslant 2 \varepsilon. 
    \end{split} \] 

    \noindent Hence the path~$f \circ c$ converges to~$\ell$ which contradicts the assumption that~$f$ is inextendible, hence the result.
\end{proof}

\begin{rem}
    This shows that property of non-trivial causality can be weakened to its continuous counterpart, as the existence of an inextendible timelike map containing~$x \in M$ implies the existence of an inextendible timelike immersion containing~$x$.
\end{rem}

\begin{lemma}\label{timelikedeformation}
Let $f : N_0 \rightarrow M$ be a timelike map and let $x \in N_0$. Let $C(a,b)$ be a cylindrical neighborhood around $f(x)$ and $g_+$ a constant metric on $C(a,b)$ for which $f$ is timelike and $g \leqslant g_+$. Then there exists a neighborhood $U$ of $f(x)$ in $C(a,b)$ such that for each $y \in U$ and each causal sub-vecttor space $Q \subset T_y M$ of dimension $q$, there exists a causal map $g : N_0 \rightarrow M$ which is equal to $f$ outside of the connected component $V$ of $f^{-1}(C(a,b))$ containing $x$ and a point $z$ in $V$ such that $g(z) = y$, $g$ is differentiable in $z$ and $\Imm(dg_z) = Q$.
\end{lemma}

\begin{proof}
For each point $w$ in the boundary of $\Imm(f)$ in $\partial C(a,b)$, the segment $[f(x), w]$ must then be timelike for $g_+$. Let $U$ be a neighborhood of $f(x)$ in $C(a,b)$ such that for all $y \in U$ and all $w$ in the boundary of $\Imm(f)$, the segment $[y, w]$ is timelike. Let $y \in U$ and $Q \subset T_y M$ a causal sub-vector space of dimension $q$. Let $E$ be the set of $C(a,b)$ made of the boundary of $\Imm(f)$ in $\partial C(a,b)$ and of the intersection of $y + Q$ with $U$. Since the set $E$ is in causal position for $g_+$, by Kirszbraun theorem, there exists a causal map $\tilde{g}$ equal to $f$ on the boundary and having $(y+Q) \cap U$ in its image. Since $g \leqslant g_+$, $\tilde{g}$ must also be causal for $g$. The map $g$ equal to $\tilde{g}$ in $V$ and to $f$ elsewhere gives us the result.
\end{proof}

\begin{rem}
This deformation result or slight variations of it will be useful at several points in our proofs. In particular, this lemma implies that having non-trivial causality means that for any point $x$ in $M$ and any causal sub-vector space of dimension $q$ $Q \subset T_y M$, there exists an inextendible causal immersion of dimension $q$ through $y$ having $Q$ as its tangent space in $y$.
\end{rem}

\subsection{Future and past of oriented causal spaces of dimension $q-1$}

One defect to the notion of causality we defined is that contrary to its Lorentzian counterpart, it lacks the existence of a future and a past. In~$\R^{p,q}$, the set of points which are in timelike position with the origin is connected as soon as~$q \geqslant 2$. In order to obtain a relevant notion of past and future, one must make a few more assumptions on the sets we consider. 

Let $M$ be a $(p,q)$-spacetime with a given time orientation. Let~$f \colon N_1 \rightarrow M$ be a causal map of dimension~$q-1$ where~$N_1$ is an oriented closed manifold. Assume that~$f$ can be extended into 
an inextendible causal map of dimension~$q$,~$\overline{f} \colon N_0 \rightarrow M$ where~$N_1$ is a closed submanifold of~$N_0$ of co-dimension one. 
Since the orientation of~$N_1$ has been fixed and the orientation of~$N_0$ stems from the time orientation of~$M$, there is a standard way to choose between the two 
connected components of~$N_0 \setminus N_1$. The interior of $N_1$ will be the future of $f$ in $\overline{f}$ and the exterior will be its past.

\begin{defn}
    Let~$f \colon N_1 \rightarrow M$ be a causal map of dimension~$q-1$ where~$N_1$ is an oriented closed manifold. A point~$x \in M$ will be said to be 
    in the \emph{future} of~$f$ (resp. its \emph{past}) when there exists an inextendible causal map~$\overline{f} \colon N_0 \rightarrow M$ of dimension~$q$ 
    which extends~$f$ and which contains~$x$ in the future of~$f$ in~$\overline{f}$. The future (resp. past) of~$f$ will be noted~$J^+(f)$ (resp.~$J^-(f)$). 
    When there is no possible confusion, we will not make the distinction between~$f$ and its image.
\end{defn}

\begin{defn}\label{openfuture}
    Let~$f \colon N_1 \rightarrow M$ be a timelike map of dimension~$q-1$ where~$N_1$ is an oriented closed manifold. A point~$x \in M$ will be said to be 
    in the \emph{open future} of~$f$ (resp. its \emph{open past}) when there exists an inextendible timelike map~$\overline{f} \colon N_0 \rightarrow M$ of dimension~$q$ 
    which extends~$f$ and which contains~$x$ in the future of~$f$ in~$\overline{f}$. The open future (resp. open past) of~$f$ will be noted~$I^+(f)$ (resp.~$I^-(f)$). 
\end{defn}

\noindent These definitions can actually be extended to a larger class of maps~$f \colon N_1 \rightarrow M$ of dimension~$q-1$. 
The only thing we need is that for each inextendible map~$\overline{f} \colon N_0 \rightarrow M$ of~$f$ of dimension~$q$, 
$N_1 \subset N_0$ separates~$N_0$ into two connected components.

\begin{defn}
    A causal map~$f \colon N_1 \rightarrow M$ of dimension~$q-1$ will be said to be \emph{splitting} if 
    for each inextendible causal map~$\overline{f} \colon N_0 \rightarrow M$ extending~$f$,~$N_0 \setminus N_1$ is the union of two connected components.
\end{defn}

\noindent We define the past and the future of~$f$ in the same manner when~$f$ is splitting.

\begin{rem}
    The notion of future and past of~$f$ depends on the orientation of~$N_1$. When taking the opposite orientation, one exchanges the past and future of~$f$.
\end{rem}

\begin{prop}
Let $f : N_1 \rightarrow M$ be a splitting map. The open future $I^+(f)$ is open in $M$.
\end{prop}

\begin{proof}
This stems directly from lemma \ref{timelikedeformation}.
\end{proof}

\subsection{Causal diamonds in the pseudo-Euclidian space}\label{causaldiamondsinminkowski}

Let us continue the study of the pseudo-Euclidian space $\R^{p,q}$ that we began in section \ref{causalityminkowski}. As shown in Proposition \ref{flatgeodesics}, any inextendible causal immersion of dimension $q$ is the graph 
of a smooth map from~$\R^q$ to~$\R^p$ which is~$1$-Lipschitz.

\begin{rem}
    Since inextendible causal maps can be approximated by inextendible smooth immersions of dimension $q$, this result extends to the 
    continuous case. Meaning that any inextendible causal map in~$\R^{p,q}$ is the graph of a continuous map from~$\R^q$ to~$\R^p$ which is~$1$-Lipschitz, or~$1$-contracting 
    when it is a timelike map.
\end{rem}

\noindent This, in particular, allows us to determine the causal maps of dimensions~$q-1$ which are splitting 
in the sense previously defined.

\begin{prop}
    An oriented causal map of dimension~$q-1$ is extendible into an inextendible causal map of dimension~$q$ 
    if and only if it is the graph of a map~$f$ defined on a submanifold of co-dimension~$1$ of~$\R^q$ which is~$1$-Lipschitz for the Riemannian distances 
    of~$\R^q$ and~$\R^p$. Furthermore, this immersion is splitting if and only if~$N_1$ separates~$\R^q$ into two connected components
\end{prop}

\begin{proof}
    The first part of the result comes from \ref{timeposition}. The second part comes immediately from the fact that any inextendible causal map extending~$f$ is a graph over~$\R^q$.
\end{proof}

\noindent In particular, we are going to study two examples. Let~$S^- = \{Q = -1\} \cap \R^{0,q}$ be the sphere of radius~$1$ 
in~$\R^{0,q}$ with the standard orientation and~$L = \R^{0,q-1} \subset \R^{0,q}$ be a hyperspace of~$\R^{0,q}$ bounding the half-space~$L \times \R_{>0}$ in~$\R^{0,q}$ with the induced orientation. Both are timelike splitting immersions of dimension~$q-1$. 
Let us first study the future of~$S^-$. Let~$S^+ = \{Q = 1\} \cap \R^{p, 0}$ be the positive sphere of radius~$1$ in~$\R^{p,0}$. The union of all lightlike segment coming from pairs of points in $S^- \times S^+$ is homeomorphic to a sphere of dimension~$p+q-1$ as it is the topological joint of two sphere of dimension $p-1$ and $q-1$. It is the boundary of a bounded domain in~$\R^{p,q}$. We will write this open space~$\Diam_{p,q}$.

\begin{figure}[h!bt]
    \begin{center}
      \includegraphics[width=.5\linewidth]{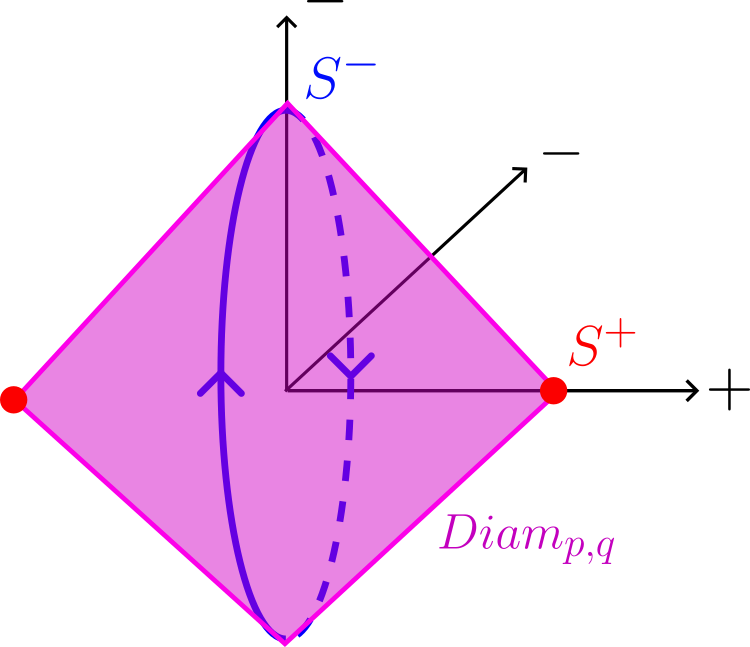}
      \end{center}
      \caption[]{The~$1$-sphere~$S_-$ and the~$0$-sphere~$S_+$.}
    \end{figure}

\begin{prop}
    The open set~$\Diam_{p,q}$ is the open future of~$S^-$ in~$\R^{p,q}$.
\end{prop}

\begin{proof}
    Let $T : \R^{p,q} \rightarrow \R^{0,q}$ be the orthogonal projection. We first wish to determine the set of points~$x$ such that~$\{x\} \cup S^-$ is in timelike position and~$T(x)$ is in the disk bounded by~$S^-$ in~$\R^{0,q}$. Using the result proven in \ref{timeposition}, 
    those points are exactly those in~$T^{-1}(D(0,1))$ such that for each~$y \in S^-$,~$[x,y]$ is timelike. This area is clearly bounded by the union of all lightlike segments from pairs of points in $S^- \times S^+$, hence the result.
\end{proof}

\begin{rem}
    The space~$\Diam_{p,q}$ can be interpreted as a generalization of the causal diamond in Lorentzian geometry, where the two points defining the diamond are replaced by an oriented timelike Möbius sphere of dimension~$q-1$.
\end{rem}

\begin{defn}
A subset $S$ of $\ein^{p,q}$ will be called a positive (resp. negative) \emph{Môbius sphere} of dimension $k$ if it is the projectivisation of the intersection of a sub-vector space of $\R^{p+1,q+1}$ of signature $\R^{k+1,1}$ (resp. $\R^{1, k+1}$) with the isotropic cone $\pazocal{C}$.
\end{defn}

When a positive (resp. negative) Möbius sphere of dimension $k$ is contained within an affine chart $\R^{p,q}$ of $\ein^{p,q}$, there exists an affine space of signature $\R^{k+1,0}$ (resp. $\R^{0, k+1}$) for which it is a sphere for the induced metric. In particular, $S^-$ is a negative Môbius sphere of dimension $q-1$.

\begin{rem}
    When~$S$ is a topological sphere in timelike position of dimension~$q-1$ in~$\R^{p,q}$, either the open future or the closed future of~$S$, depending on its orienation, is bounded. It will be called a causal diamond. When $S$ is a Möbius sphere, the causal diamond will be said to be \emph{flat}.
\end{rem}

\noindent Let's now move on to the space~$L$.

\begin{prop}
    The future~$I^+(L)$ of~$L$ is the set of points~$x \in \R^{p,q}$ such that~$T(x)$ is in~$L \times \R_{>0}$ 
    and there exists a unique~$y \in L$ such that the segment~$[x,y]$ is timelike and orthogonal to~$L$.
\end{prop}

\begin{proof}
    Assume that~$T(x) \in L \times \R_{>0}$ and that the exists~$y \in L$ such that~$[x,y]$ is timelike and orthogonal to~$L$. The first condition tells us 
    that~$x$ is going to be in the future of~$L$ rather than is its past. Since~$[x,y]$ is orthogonal to~$L$, the vector space~$L \oplus x$ is negative of dimension~$q$, meaning 
    that~$L \cup \{x\}$ is in timelike position, hence~$x \in I^+(L)$. Conversely, if~$x$ is in the future of~$L$, then for each~$y \in L$,~$[x,y]$ is timelike. The orthogonal projection $y$ of $x$ onto $L$ is the unique point for which $[x,y]$ is orthogonal to $L$. Since $[x,y]$ must be causal, this gives the result.
\end{proof}

\begin{figure}[hbt]
    \begin{center}
      \includegraphics[width=.7\linewidth]{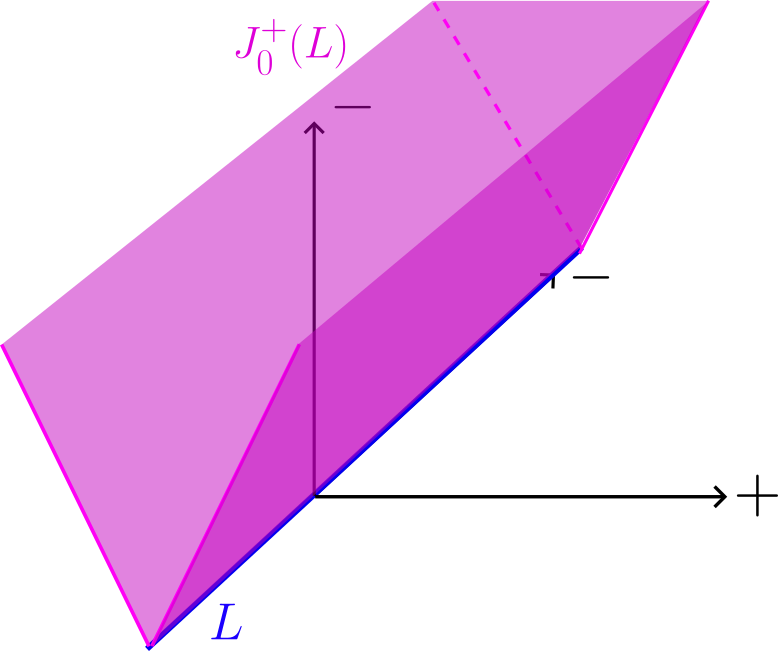}
      \end{center}
      \caption[]{The timelike line~$L$ and its future~$I^+(L)$.}
    \end{figure}

\begin{rem}
In particular, combining the orientation of $L$ with the time orientation of $\R^{p,q}$ gives a time orientation on the Lorentzian space $L^{\perp}$ and $I^+(L)$ is equal to $L + I^+(0)$ where $I^+(0)$ is the future of the origin in $L^{\perp}$.
\end{rem}

\begin{rem}
    This second example is to be compared with the open future cone of a point in the Lorentzian Minkowski space. In Lorentzian geometry, 
    both this space and the Lorentzian diamond are actually equivalent, in the sense that there exists a conformal transformation taking the first one to the second. 
    The same is true in higher signature.
\end{rem}

\begin{prop}
    The space~$\Diam_{p,q}$ is conformally equivalent to~$I^+(L)$.
\end{prop}

\begin{proof}
    Let $e_{p+1},\cdots, e_{p+q-1}$ be an orthonormal basis of $L$ completed into an orthonormal basis $e_1,\cdots,e_{p+q}$ of $\R^{p,q}$ with $Q(e_{p+q}) < 0$. We are going to use the map

    \[\varphi \colon x \in \Diam_{p,q} \longmapsto - \frac{x-e_{p+q}}{Q(x-e_{p+q})}.\]

    \noindent The map~$\varphi$ is the restriction of an inversion in~$\widehat{\ein}^{p,q}$, in particular it is a conformal transformation. We now only need to show that~$\varphi$ maps~$\Diam_{p,q}$ to~$I^+(L)$. 
    Since the notion of causality only depends on the conformal class of the metric, time-orientation preserving conformal maps send future to future. Thus it is enough to 
    show that~$\varphi$ takes~$S^-$ to~$L$, which is clear as its restriction to~$\R^{0,q}$ is the usual map taking the unit disk to the half-space, hence the result.
\end{proof}

\noindent This equivalence is particularly interesting, as there is a nice conformal model for~$J^+_0(L)$.

\begin{prop}\label{causaldiamonds}
   ~$J^+_0(L)$ is conformally equivalent to~$\h^p \times \h^q$.
\end{prop}

\begin{proof}
    Let us consider the half-plane model of~$\h^q \simeq \R^{q-1} \times \R_{>0}$ with~$\R^{q-1}$ identified to~$L$ and~$\h^p \subset \R^{p,q}$ 
    taken to be the space~$\{Q = -1\} \cap (\R^{p,0} \oplus \R^{0,1})$ where~$\R^{0,1}$ is the orthogonal of~$L$ in~$\R^{0,q}$.The map we are going to consider is the following, 

    \[\psi \colon (x, (y, t)) \in \h^p \times \h^q \longmapsto t x + y.\]

    \noindent  It is clear that~$\psi$ takes~$\h^p \times \h^q$ to~$I^+(L)$, we now have to show that it is a conformal transformation. Let us first show that the image by~$\psi$ of the spaces~$\h^p \times \{*\}$ and~$\{*\} \times \h^q$ are orthogonal. Let~$(x,y,t) \in \h^p \times \h^q$. 
    The image by~$\psi$ of the set~$\h^p \times \{(y,t)\}$ is one of the two hyperbolic sheets of the affine pseudo-Euclidian space given by the orthogonal of~$L$ in~$y$.
    The image of~$\{x\} \times \h^{q}$ is in the vector space~$L \oplus x$ on which~$Q$ is definite negative. We see that those two affine spaces meet orthogonally in~$tx + y$.

    In order to show that~$\psi$ is a conformal transformation, we now only have to show that~$\psi|_{\h^p \times \{(y,t)\}}$ and~$\psi|_{\{x\} \times \h^q}$ are both conformal transformations on their image with the same conformal factor. 
    It is then easy to see that 

    \[\psi_{*} g_{\h^p \times \h^q} = \frac{1}{t} g_{\R^{p,q}},\]

    \noindent hence the result.
\end{proof}

\noindent From this result, we can deduce two corollaries.

\begin{coro}
    The pseudo-Riemannian space~$\h^p \times \h^q$ is conformally flat
\end{coro}

\begin{coro}
    The diamond~$\Diam_{p,q}$ is conformally equivalent to~$\h^p \times \h^q$. In particular,~$\Diam_{p,q}$ gives an example of a proper open
    subspace of~$\ein^{p,q}$ which is conformally homogenous. 
\end{coro}

\begin{rem}
    Conjectures on the proper open spaces of flag manifolds predicted that such domains in~$\ein^{p,q}$ should be classifiable (see \cite{Zimmer}); In \cite{chalumeau2024properquasihomogeneousdomainseinstein}, Chalumeau and Galiay proved that the flat causal diamonds are the only proper quasi-homogenous domains of~$\ein^{p,q}$.
\end{rem}

The sphere~$S^-$ seen as a subset of $\ein^{p,q}$ via the conformal embedding $\R^{p,q}\subset \widehat{\ein}^{p,q}$ can be taken to any timelike Möbius sphere of~$\ein^{p,q}$ of dimension~$q-1$ via the conformal transformations of~$\widehat{\ein}^{p,q}$. Thus, any two causal diamonds can be linked by a conformal transformation of~$\widehat{\ein}^{p,q}$. In particular, the past of~$S_-$ in $\widehat{\ein}^{p,q}$ is also a causal diamond in any affine chart which contains it.

As we have seen before, one of the way higher signature differs from Lorentzian signature is that not every timelike map can be extended into an inextendible timelike map. 
One consequence is that when taking an open set~$\Omega$ of a spacetime~$M$ (for instance~$M = \R^{p,q}$), it is possible that some timelike maps could be 
inextendible in~$\Omega$ but not in~$M$, and furthermore that this map could not be extended into an inextendible timelike map in~$M$. In other words, for~$x,y$ in~$\Omega$, 
it is possible that~$\{x,y\}$ is in timelike position in~$\Omega$ but not in~$M$. 

\begin{exmp}
    Assume $\R^{p,q} = \R^{1,2}$ and let~$U = D(0,2) \setminus \overline{D(0,1)}$ be an open ring in~$\R^2$. Let~$\Omega = T^{-1}(U)$. For each~$x,y \in \Omega$,~$\{x,y\}$ is 
    in timelike position in~$\Omega$.
\end{exmp}

\begin{proof}
    Let $\pi : \tilde{U} \simeq \R \times (0,1) \rightarrow U$ be the universal covering of $U$, $a>0$, $b \in \R^{1,0}$ and let $f : (x,t) \mapsto at +b + \pi(t,x) $. By the right $a$ close to zero and $b$, one may ensure that $f$ is a timelike immersion of dimension $2$ containing any two pair of points in $T^{-1}(\Omega)$. The map $f$ is clearly inextendible in $T^{-1}(\Omega)$, hence the result.
\end{proof}

\noindent Notice that in the previous proof,~$f$ seen as a map from $\tilde{U}$ to $\R^{p,q}$ is not inextendible in~$\R^{p,q}$ as the path $c : s \in (0,1)\rightarrow (t,s) \in \tilde{U}$ is inextendible in $\tilde{U}$ while $f \circ c$ converges in $\R^{p,q}$. It also cannot be extended into an inextendible timelike map, since it is not the graph of a map from 
a subset of~$\R^q$ into~$\R^p$. To end this subsection, we will give a property on open subspaces~$\Omega$ of~$\R^{p,q}$ which 
guarantees that~$\Omega$ and~$\R^{p,q}$ share the same notion of causality.

\begin{prop}
    Let~$U$ be an open subset of $\R^q$. The set~$\Omega = T^{-1}(U)$ has the same causality as~$\R^{p,q}$ if~$U$ is convex in~$\R^q$.
\end{prop}

\begin{proof}
    Assume~$U$ is convex and let~$\tilde{f} \colon N_0 \rightarrow \Omega$ be an inextendible timelike immersion. 
    For the same reasons as before,~$T \circ \tilde{f}$ must be both surjective and a covering on~$U$. However since~$U$ is convex it 
    must be simply connected, meaning that $T \circ f$ is a diffeomorphism and~$\tilde{f}$ can be seen as the graph of a map~$f \colon U \rightarrow \R^p$ which verfies that for all 
   ~$x \in U$,~$||df_x|| < 1$. This means that for each smooth path~$c$ of~$U$, 

    \[L(f(c)) < L(c).\]

    \noindent Since~$U$ is convex, for each~$x,y \in U$, the segment~$[x,y]$ is in~$U$ and minimises the length of the curves~$c$ from~$x$ to~$y$, 
    meaning that~$L([x,y]) = d(x,y)$. This means that we must have

    \[d(f(x), f(y)) < d(x,y) \]

    \noindent for each~$x,y \in U$, meaning that~$f$ is a~$1$-contracting map defined on~$U \subset \R^q$. By \ref{timeposition}, 
    this means that~$f$ can be extended into a~$1$-contracting map defined on the whole of~$\R^q$, hence the result.

    Conversely, assume that~$U$ is not convex. It will always be possible to build a map 
   ~$f \colon U \rightarrow \R^p$ which verifies~$\forall x \in U, ||df_x|| < 1$ but is not~$1$-contracting.
\end{proof}

\section{Global hyperbolicity in higher signature}

\subsection{Cauchy surfaces}

This section will be dedicated to a generalization 
of the notion of global hyperbolicity to the context of pseudo-Riemannian spacetimes 
of higher signature. In Lorentzian geometry, being globally hyperbolic can be characterized 
by each one of those three equivalent properties :

\begin{enumerate}
    \item The spacetime~$M$ admits a Cauchy surface, 
    \item The spacetime $M$ admits a Cauchy time function~$T \colon M \rightarrow \R$, 
    \item For each~$x,y \in M$, the space~$\Cc(x,y)$ of non-parametrized future-oriented causal curves from~$x$ to~$y$ is compact for the uniform topology.
\end{enumerate}

\noindent We are going to start by generalizing the notion of Cauchy surface.

\begin{defn}
    Let~$M$ be a spacetime of signature~$(p,q)$. A subset~$S$ of~$M$ is said to be a \emph{Cauchy surface} if for each 
    inextendible causal map~$f \colon N_0 \rightarrow M$ of dimension~$q$, there exists a unique~$x \in N_0$ such that~$f(x) \in S$.
\end{defn}

\begin{rem}
    In Lorentzian geometry, the usual definition for a Cauchy surface only requires that every inextendible timelike curve meets the set~$S$ exactly once, thus allowing the existence of Cauchy surfaces containing parts of lightlike geodesics. For our generalization, we will restrict ourselves to acausal Cauchy surfaces.
\end{rem}

\noindent We have already seen an example of Cauchy surfaces in the pseudo-Euclidian case in Proposition \ref{EuclidianGH}.

\begin{prop}
    Let~$S = \R^{p,0} \subset \R^{p,q}$. Then~$S$ is a Cauchy surface of~$\R^{p,q}$.
\end{prop}

\noindent This result is comforting as the Lorentzian Minkowski space is the basic example of a globally hyperbolic Lorentzian space. We can actually exhibit more Cauchy surfaces in~$\R^{p,q}$.

\begin{prop}
    Let $g : \R^p \rightarrow \R^q$ be a $k$-Lipschitz map with $k < 1$. The graph of $g$ in $\R^{p,q}$ is a Cauchy surface.
\end{prop}

\begin{proof}
    Any inextendible causal map is the graph of a $1$-Lipschitz map from $\R^q$ to $\R^p$. Let $f$ be such a map. The graph of $f$ inersects the graph of $g$ in a point $(a,b)$ if and only if there exists a point $x \in \R^p$ and a point $y \in \R^q$ such that $(a,b) = (x, g(x)) = (f(y),y)$. This happens if and only if $x = f(y)$ and $g(x) = y$, meaning that $g(f(y)) = y$. Therefore there is a correspondance between the intersections of the graphs of $f$ and $g$ and the fixed points of the map $g \circ f : \R^q \rightarrow \R^q$. However $g \circ f$ is $k$-Lipschitz with $k<1$ and $\R^q$ is complete, thus it must have exactly one fixed point. Therefore the graph of $f$ intersects the graph of $g$ in exactly one point, hence the result.
\end{proof}

\noindent In Lorentzian geometry, every non-causal spacetime automatically fails to be globally hyperbolic. Indeed, a timelike loop~$c \colon \s^1 \rightarrow M$ can be seen as an inextendible periodic timelike path~$\hat{c} \colon \R \rightarrow M$, if it encounters a subset of $M$, it does so an infinite number of times, hence the result. Mimicking this, we define the following notion of causal pseudo-Riemannian space:

\begin{defn}
    A spacetime of signature~$(p,q)$ is said to be \emph{causal} if there does not exist an inextendible causal map~$f \colon N_0 \rightarrow M$ of dimension~$q$ where~$N_0$ is not simply connected.
\end{defn}

\begin{lemma}
    If~$M$ admits a Cauchy surface~$S$, it is causal.
\end{lemma}

\begin{proof}
    Assume there exists an inextendible causal map~$f \colon N_0 \rightarrow M$ where~$N_0$ is not simply connected and let~$\tilde{f} \colon \tilde{N_0} \rightarrow M$ be the map induced by~$f$ on the universal covering~$\tilde{N_0}$ of~$N_0$. Since~$\tilde{f}$ is also an inextendible causal map, there exists a point~$x_1$ in~$\tilde{N_0}$ such that~$\tilde{f}(x_1) = f(\pi(x_1)) \in S$. However since~$N_0$ is not simply connected, there exists a point~$x_2 \neq x_1$ such that~$\pi(x_1) = \pi(x_2)$, meaning that~$\tilde{f}(x_1) = \tilde{f}(x_2)$. The point~$x \in \tilde{N_0}$ such that~$\tilde{f}(x) \in S$ is supposed to be unique, hence the result.
\end{proof}

\noindent In most cases, it will be easier to check that a subset~$S \subset M$ is a Cauchy surface by verifying the property 
on smooth inextendible causal immersion of dimension $q$. Let us prove that this is enough to get the property on continuous inextendible causal maps. For this, we will need the following lemma:

\begin{lemma}
    Let~$S \subset M$ be a subset verifying the Cauchy property for smooth maps. Then~$S$ is closed in~$M$.
\end{lemma}

\begin{proof}
    Assume that~$S$ is not closed and let~$y \in \overline{S} \setminus S$. Let~$f \colon N_0 \rightarrow M$ be a smooth inextendible timelike map 
    such that there exists~$x \in N_0$ with~$f(x) =y$. There must exist a unique~$x_1 \in N_0$ such that~$f(x_1) \in S$. However, by using lemma \ref{timelikedeformation}, one can slightly deform~$f$ into an inextendible timelike immersion~$\tilde{f}$ such that~$\tilde{f}(x) \in S$ and~$\tilde{f} = f$ outside of a small neighborhood around~$x$. The map~$\tilde{f}$ thus verifies that~$\tilde{f}(x_1)$ and~$\tilde{f}(x)$ are both in~$S$ which is impossible, hence the result.
\end{proof}

\begin{prop}
    Let~$S \subset M$ be a subset verifying the Cauchy property for smooth maps. Then~$S$ also verifies the property for continuous maps, meaning that~$S$ is a Cauchy surface.
\end{prop}

\begin{proof}
    Let~$f \colon N_0 \rightarrow M$ be an inextendible continuous causal map and let's assume the image of~$f$ does not intersect~$S$. Since~$M \setminus S$ is open, one can uniformaly approximate~$f$ by a smooth inextendible causal immersion~$f_1$ of dimension $q$ which also does not meet~$S$, which is impossible. Hence, there must exist a~$x$ in~$N_0$ such that~$f(x) \in S$.

    \indent Assume now that there exists two distinct points~$x_1$ and~$x_2$ in~$N_0$ such that~$f(x_1)$ and~$x(x_2)$ are both in~$S$. Then one can uniformaly approximate~$f$ by a smooth inextendible causal immersion~$f_2$ of dimension $q$ such that~$f_2(x_1) = f(x_1)$ and~$f_2(x_2) = f(x_2)$ which are both in~$S$. This is impossible as~$S$ verifies the Cauchy property for smooth maps, hence the result.
\end{proof}

\subsection{Structure of Cauchy surfaces}

One interesting property about Cauchy hypersurfaces in Lorentzian geometry is that their topology is determined up to homeomorphism (or diffeomorphism when both surfaces are smooth). The proof of this result relies on the existence of timelike vector fields on Lorentzian spaces, which are always integrable. In signature~$(p,q)$, there always exists timelike distributions of~$q$-vector spaces, however those distributions sometimes fail to be integrable. Examples of spacetimes which do not admit a timelike foliation of dimension~$q$ will be given later. We will prove first a few results under the assumption of the existence of such a foliation.

\begin{prop}
    Let~$\pazocal{F}$ be a timelike (resp. causal) foliation of dimension~$q$ on a spacetime~$M$. Every leaf of~$\pazocal{F}$ is an inextendible timelike (resp. causal) immersion of dimension~$q$.
\end{prop}

\begin{proof}
    Let~$F$ be a leaf of~$\pazocal{F}$. The inclusion~$i \colon F \rightarrow M$ is a timelike (resp. causal) embedding of~$F$ into~$M$. Let~$D$ be the distribution associated with~$\pazocal{F}$ and assume there exists a path~$c \colon (a,b) \rightarrow F$ and a point~$x$ in~$M \setminus F$ such that~$c(t) \longrightarrow_{t \rightarrow b} x$. The distribution~$D$ is locally integrable around~$x$ which is in the closure of~$F$, meaning that~$x$ must also be in the leaf~$F$, hence the result.
\end{proof}

\begin{defn}
    A foliation is said to be \emph{regular} if there exists a submersion~$\varphi \colon M \rightarrow P$ such that the leaves of~$\pazocal{F}$ are exactly the level sets of~$\varphi$.
\end{defn}

\begin{prop}
    Let~$M$ be a~$(p,q)$-spacetime admitting a smooth Cauchy surface~$S$ and a causal foliation~$\pazocal{F}$ of dimension $q$. Then~$\pazocal{F}$ is a regular foliation and $S$ is homeomorphic to $M/\pazocal{F}$.
\end{prop}

\begin{proof}
    Let~$x \in M$ and let~$F$ be the leaf of~$\pazocal{F}$ containing~$x$. Since~$F$ is an inextendible causal immersion of dimension~$q$, there exists a unique point~$\varphi(x)$ both in~$F$ and in~$S$. The map~$x \in M \mapsto \varphi(x) \in S$ is a smooth submersion for which the level sets are exactly the leaves of~$\pazocal{F}$, hence the result.
\end{proof}

\begin{prop}
    Let~$M$ be a~$(p,q)$-spacetime with a causal foliation~$\pazocal{F}$ of dimension~$q$ and two Cauchy surfaces~$S_1$ and~$S_2$. Then~$S_1$ and~$S_2$ are homeomorphic.
\end{prop}

\begin{proof}
    The sets $S_1$ and $S_2$ are both homeomorphic to $M/\pazocal{F}$, hence the result.
\end{proof}

Though Cauchy surfaces are defined as general subsets a priori, one may show that this definition gives restrictions on their topology within $M$.

\begin{prop}
Let $M$ be a spacetime with non-trivial causality and let $S$ be a Cauchy surface in $M$. Then $S$ is a submanifold of dimension $p$ which is locally Lipschitz. 
\end{prop}

\begin{proof}
Let $x$ be a point in $S$ and $C(a,b)$ a cylindrical neihborhood in $T_x M$ centered on $x$ endowed with a constant metric $g_+$ such that $g < g_+$ on $C(a,b)$. Let $f : N_0 \rightarrow M$ be an inextendible timelike map containing $x$ such that in a small neighborhood $U$ of $x$ in $C(a,b)$, the image of $f$ is a timelike vector space $V$ for $g_+$ of dimension $q$. Finally, let $D$ be a flat disk in $V$ centered on $x$. 

Any causal disk in $\Diam_{g_+}(\partial D)$ with boundary is also causal for $g$ since $g < g_+$ and as such, it can be glued to $f$ along $\partial D$ into an inextendible causal map of $M$. Since $S$ is a Cauchy surface and $f$ already inersects $S$ in the unique point $x$, any causal disk in $\Diam_{g_+}(\partial D)$ with boundary $\partial D$ must then intersect $S$ exactly once. Assume that there exists two point $a,b$ in $S \cap \Diam_{g_+}(\partial D)$ such that the segment $[a,b]$ is causal for $g_+$. Then since both $a$ and $b$ are in $\Diam_{g_+}(\partial D)$, the set $\partial D \cup \{a,b\}$ is in causal position in $\Diam_{g_+}(\partial D)$, therefore there exists a causal disk in $\Diam_{g_+}(\partial D)$ with boundary $D$ containg both $a$ and $b$. However both $a$ and $b$ were assumed to be in $S$, thus the gluing of the disk with $f$ would intersects $S$ in two different points which is not possible.

From this we deduce that for any two points $a,b$ in $S \cup \Diam_{g_+}(\partial D)$, the segment $[a,b]$ is spacelike, meaning that $S$ is locally the graph of a $1$-Lipschitz map from an open space of $\R^p$ to $\R^q$, hence the result.
\end{proof}

\subsection{Cauchy time functions}

We will generalize the notion of Cauchy time functions in the same manner.

\begin{defn}
    Let~$M$ be a spacetime of signature~$(p,q)$. A \emph{Cauchy time function} is a map~$T \colon M \rightarrow N$ where~$N$ is a~$q$-dimensional manifold such that for each~$y \in N$,~$T^{-1}(y)$ is a Cauchy surface.
\end{defn}

\begin{rem}
    An equivalent formulation is that for each inextendible causal map~$f \colon N_0 \rightarrow M$, the map~$T \circ f \colon N_0 \rightarrow N$ is a homeomorphism.
\end{rem}

\noindent Once again, the Cauchy time function property only has to be checked for smooth maps.

\begin{prop}
    Let~$T \colon M \rightarrow N$ be a map such that for each smooth inextendible causal immersion~$f \colon N_0 \rightarrow M$ of dimension $q$, the map~$T \circ f$ is a homeomorphism. Then~$T$ is a Cauchy time function.
\end{prop}

\begin{proof}
    Since any level set of $T$ meets all smooth inextendible causal immersions of dimension $q$, all level sets of $T$ are Cauchy surfaces and $T$ is a Cauchy time function by definition.
\end{proof}

\begin{coro}
    Let~$T \colon M \rightarrow N$ be a Cauchy time function on $M$. Then every base space of an inextendible causal map is homeomorphic to~$N$. More precisely, every inextendible causal map can be seen as a section of the map~$T$.
\end{coro}

\noindent This gives us an easy way to verify that a spacetime does not have a Cauchy time function. For instance,~$\R^{p,q} \setminus \{0\}$ does not have a Cauchy time function since some of the inextendible timelike maps are homeomorphic to~$\R^q$ while others are homeomorphic to~$\R^q \setminus \{0\}$. In \ref{EuclidianGH}, we have shown that~$T \colon \R^{p,q} \rightarrow \R^q$ is a Cauchy time function. By using the same arguments, we can exhibit a very large class of examples of spaces admitting a Cauchy time function.

\begin{defn}
Let $\pi : (M, g_M) \rightarrow (N, g_N)$ be a submersion between two Riemannian manifolds $M$ and $N$. We will say that $\pi$ is a \emph{Riemannian submersion} if for each $x \in M$, the restriction of $df_x$ to $\Ker(d \pi)^{\perp}$ is an isometry between $\Ker(d \pi)^{\perp}$ and $T_{\pi(x)} N$. Furthermore, when $\pi$ is a smooth bundle, we will say it is a \emph{Riemannian bundle}.
\end{defn}

\begin{prop}\label{mainexample}
    Let~$T \colon (M^{p+q}, g_M) \rightarrow (N^q, g_N)$ be a complete Riemannian bundle where~$N$ is simply connected and let~$g$ be the~$(p,q)$ metric defined on~$M$ by the splitting~$Ker(d T) \oplus (Ker(d T))^{\perp}$ where each fiber is positive and its orthogonal is negative. Then the map~$T \colon M \rightarrow N$ is a Cauchy time function.
\end{prop}

\begin{proof}
    Let us first show that the map~$T \circ f$ is a local diffeomorphism. Since the projection~$T$ removes the positive coordinates, the map~$T \circ f$ increases the distances between the manifolds~$(N_0, -f^* g)$ and~$(N, g_N)$, i.e for each~$x \in N_0, v \in T_x N_0$, 

    \[ g_N(d(T \circ f)(v)) \geqslant - f^* g(v). \] 

    \noindent This tells us that~$T \circ f$ is a local diffeomorphism ; in particular, its image~$\Imm(T \circ f)$ is open in~$N$.

    Let's assume that~$T \circ f$ is not surjective let~$c \colon (a,b) \rightarrow N$ be a smooth path of speed~$1$ in~$\Imm(T \circ f)$ converging to~$y \in \partial \Imm(T \circ f)$ in~$b$. For each~$x \in M$ and~$v \in T_x M$ causal, we have that~$g_M(v) \leqslant 2 g_N(dT(v))$, since~$g(v) \leqslant 0$. In particular, if~$\tilde{c}$ is a path in~$N_0$ such that~$T \circ f \circ \tilde{c} = c$, we have that the speed of~$f(\tilde{c})$ for~$g_M$ is less than~$2$. Since~$(M, g_M)$ is complete,~$f(\tilde{c})$ must converge, and since~$f$ is inextendible so must~$\tilde{c}$. However we have assumed that~$y \notin \Imm(T \circ f)$, hence the contradiction ; the map~$T \circ f$ must be surjective.

     Let us now show that~$T \circ f$ is a covering on~$N$. Let~$y \in N$ and let~$c \colon [0,1] \rightarrow N$ be a smooth path such that~$c(0) = y$. Let~$x \in N_0$ such that~$T \circ f(x) = y$ and~$\tilde{c} \colon I = (0,b) \rightarrow N_0$ such that~$T \circ f(\tilde{c}) = c$ with $b$ maximal. Assume that~$b < 1$. By the same reasoning as before, the smooth path~$f(\tilde{c})$ must converge in~$M$. Since~$f$ is inextendible, the path~$\tilde{c}$ must also converge in~$N_0$. However since~$T \circ f$ is a local diffeomorphism,~$\tilde{c}$ can then be defined further than~$b$, meaning that~$b = 1$. The map~$T \circ f$ must be a covering from~$N_0$ to~$N$ and since~$N$ is simply connected, the map~$T \circ f$ must be a diffeomorphism from~$N_0$ to~$\R^q$, hence the result. 

     Since it is enough to prove the Cauchy property on smooth maps, we have shown that~$T \colon M \rightarrow N$ is a Cauchy time function.
\end{proof}

\begin{rem}
    Since the notion of inextendible causal map does not change when taking another metric in the same conformal class, the previous result applies to spaces conformally equivalent to complete pseudo-Riemannian bundles.
    In particular, the proposition applies to the case where~$P$ is a complete Riemannian manifold,~$N$ a complete simply connected Riemannian manifold, and $M = P \times N$.  For instance, this tells us that when~$q \geqslant 2$ and for each decomposition~$\widehat{\ein}^{p,q} \simeq \s^p \times \s^q$, the projection onto~$\s^q$ is a Cauchy time function. We have shown in section \ref{causaldiamondsinminkowski} that the future of a timelike Möbius sphere of dimension~$q-1$ in~$\R^{p,q}$ is conformally equivalent to~$\h^p \times \h^q$. This means in particular in those spaces, the projection onto~$\h^q$ is a Cauchy time function.
\end{rem}

\begin{rem}
    When~$T \colon M \rightarrow N$ is a Cauchy time function, the inextendible causal maps of~$M$ are all equivalent to global sections~$s \colon N \rightarrow M$ of the map~$T$. In particular, taking a complete Riemannian bundle which does not admit a section gives us an example of a spacetime with trivial causality. The Hopf bundle~$\s^3 \rightarrow \s^2$ is such a space. Non-trivial complete Riemannian bundles with non-trivial causality include but are not limited to tangent bundles on a simply connected Riemannian space endowed with the corresponding Sasaki metric. It is not known whether the topological obstruction to the existence of a section is the only obstruction to non-trivial causality on a complete Riemannian bundle.
\end{rem}

\begin{coro}\label{productcase}
    Let~$P$ be a complete Riemannian manifold and~$N$ a simply connected Riemannian manifold. The projection on the second factor on the pseudo-Riemannian product~$M = P \times N$ is a Cauchy time function.
\end{coro}

\begin{proof}
    In order to apply proposition \ref{mainexample}, we need to show that in the case of a product, we only need the completeness of the fibers to conclude. Since~$N$ is Riemannian, there exists a smooth map~$f \colon N \rightarrow \R$ such that~$f \geqslant 1$ and~$(N, f g_N)$ is complete. The manifold~$(M, fg_P - fg_N)$ is conformally equivalent to~$(M, g_P - g_N)$ and is a Riemannian submersion for~$M \rightarrow N$. Furthermore, we have that 

    \[ g_P + f g_N \leqslant fg_P + fg_N, \]

    \noindent and the Riemannian manifold~$(M, g_P + f g_N)$ is complete as the product of two complete manifolds, thus~$(M, fg_P + fg_N)$ is also complete and we can apply proposition \ref{mainexample} which gives us the result.
\end{proof}

\noindent It is a well known result that every Riemannian submersion~$T \colon M \rightarrow N$ where~$M$ is complete is actually a fiber bundle. Conversely, for every fiber bundle~$T \colon M \rightarrow N$, it is possible to find Riemannian metrics on~$M$ and~$N$ such that~$M$ is complete and~$T$ is a Riemannian submersion. In our case, this tells us that the topology of any smooth fiber bundle with simply connected base is possible for a space $M$ with Cauchy time function. This clashes with the Lorentzian case where every spacetime with a Cauchy time function is a trivial bundle over $\R$.

\subsection{Restriction on the topological structure by Cauchy time functions}

As we have seen in Proposition \ref{topologicalrigidity}, being globally hyperbolic in the Lorentzian case implies strong topological caracteristics on the manifold~$M$. In order to obtain the same kinds of results in higher signature, one has to require the existence of a causal foliation of dimension~$q$.

\begin{prop}\label{timelikefoliation}
    Let~$T \colon M \rightarrow N$ be a Cauchy time function on a space $M$ with a timelike (resp. causal) foliation~$\pazocal{F}$. Then the bundle~$T \colon M \rightarrow N$ is homeomorphic to the product~$S \times N$ where~$S$ is one of the fibers. In particular, all the fibers are homeomorphic. Furthermore, the sets~$\{*\} \times N$ form a timelike (resp. causal) foliation of~$M$.
\end{prop}

\begin{proof}
    Let~$x \in M$,~$F_x$ the leaf of~$F$ containing~$x$ and~$\varphi(x)$ the only point both in~$F_x$ and~$S$. The map

    \[x \in M \longmapsto (\varphi(x), T(x)) \in S \times N\]

    \noindent is a fiber bundle isomorphism, hence the result.
\end{proof}

\noindent The examples given all have the topological structure of a smooth fiber bundle~$T \colon M \rightarrow N$. We are going to show that this is actually implied by the definition of Cauchy time function when we assume that~$T$ is smooth and that at least one of the level sets is compact.

\begin{defn}\label{subbundle}
    Let~$T \colon M \rightarrow N$ be a smooth Cauchy time function on $M$ and let~$\gamma \colon I = (a,b) \rightarrow N$ be a smooth path on~$N$. We define the pullback~$\gamma^* M$ as the space 

    \[\bigsqcup_{t \in I} T^{-1}(\gamma(t)),\]

    \noindent along with the projection~$\gamma^* T \colon \gamma^* M \rightarrow I$. This defines a smooth manifold for which the tangent space in~$x \in \gamma^* M$ is~$T_x \gamma^* M = (Ker (dT)) \oplus dT^{-1}(\dot{\gamma}(t))$. This also gives us a Lorentzian metric on~$\gamma^* M$ induced by the metric on~$M$.
\end{defn}

\noindent Those spaces we built using the structure on~$M$ keep some of its properties.

\begin{prop}
    Let~$\gamma \colon I \rightarrow N$ be a smooth path on~$N$. The map~$\gamma^* T \colon \gamma^* M \rightarrow I$ is a time function on the Lorentzian space~$\gamma^* M$.
\end{prop}

\begin{proof}
    Let~$c \colon J \rightarrow \gamma^* M$ be a timelike path. We are going to show that the map~$\gamma^* T \circ c \colon J \rightarrow I$ is injective. Let~$t \in J$. Since the vector~$\dot{c}(t) \in T_{c(t)} M$ is timelike, we know from lemma \ref{timelikedeformation} that there exists an inextendible timelike immersion~$f \colon N \rightarrow M$ such that~$f$ contains~$c(t')$ for~$t' \in ]t - \varepsilon, t + \varepsilon[$. Since~$T \circ f$ is a diffeomorphism, this implies that the map~$\gamma^* T \circ c$ is injective on~$]t - \varepsilon, t + \varepsilon[$, meaning that~$\gamma^* T \circ c$ is locally injective. Since it is a map between two intervals, it must be globally injective, hence the result.
\end{proof}

\noindent We would like to show that those projections are not only time functions, but also Cauchy time functions in the Lorentzian sense. However, we are unable to do so without the key assumption that~$T$ admits at least one level sets which is compact. It is not known wether we can get rid of this hypothesis.

\begin{theo}\label{subbundleGH}
    Let~$T \colon M \rightarrow N$ be a smooth Cauchy time function on~$M$ such that there exists~$x \in N$ such that~$T^{-1}(x)$ is compact and let~$\gamma \colon I = (a,b) \rightarrow N$ be a smooth path. Then the map~$\gamma^* T \colon \gamma^* M \rightarrow I$ is a Cauchy time function on the Lorentzian space~$\gamma^* M$.
\end{theo}

\begin{proof}
    Let~$(T_1, T_2) \subset I$ and let $c \colon (T_1, T_2) \rightarrow \gamma^* M$ be an inextendible timelike path such that~$\gamma^* T \circ c = \gamma|_{(T_1, T_2)}$. Assume that~$T_2 < b$ and let~$(t_n)$ be an increasing sequence of~$(T_1, T_2)$ converging to~$b$. For each~$n$, let~$f_n \colon N \rightarrow M$ be an inextendible timelike immersion containing~$c(t_n) = f_n(\gamma(t_n))$ and let~$p_n = f_n(x)$. Since~$T^{-1}(x)$ is compact, we can assume that~$(p_n)$ converges to~$p_{\infty} \in T^{-1}(x)$.

    \indent Since~$p_{\infty}$ is an accumulation point of the~$f_n$, by lemma \ref{limit} we can build a causal inextendible lmit map~$f_{\infty} \colon N \rightarrow M$ containing~$p_{\infty}$. The map~$f_{\infty}$ in defined on the whole of~$N$ since~$T$ is a Cauchy time function. Let~$c_{\infty} = f_{\infty}(T_2)$. Since~$t_n$ converges to~$T_2$,~$c(t_n)$ is contained in the image of~$f_n$ and~$c_{\infty}$ is an accumulation point of the~$f_n$, we must have that~$c(t_n)$ converges to~$c_{\infty}$ which is a contradiction as~$c$ was assumed to be inextendible, hence the result.
\end{proof}

\noindent This result has a few immediate consequences.

\begin{coro}
    When~$T$ admits at least one compact level set, the map~$T \colon M \rightarrow N$ is a smooth fiber bundle on~$N$. 
\end{coro}

\begin{proof}
    It is enough to show that for each smooth path~$\gamma \colon I \rightarrow M$, the map~$\gamma^* T \colon \gamma^* M \rightarrow I$ is a fiber bundle. However it is a Cauchy time function, meaning that it must be diffeomorphic to the trivial bundle~$S \times I$ where~$S$ is a level set of~$T$, hence the result.
\end{proof}

\begin{rem}
    In particular, all of the level sets are diffeomorphic, and thus compact.
\end{rem}

\noindent It is possible to improve proposition \ref{subbundleGH} by taking an immersion in~$N$ instead of a smooth path.

\begin{defn}
    Let~$f \colon N' \rightarrow N$ be a immersion from a smooth manifold~$N'$ of dimension~$q' \leqslant q$ to~$N$. In the same manner as definition \ref{subbundle}, one can define the space~$f^* M$ along with a projection~$f^* T \colon f^* M \rightarrow N'$. The metric on~$M$ endowes~$f^* M$ with a pseudo-Riemannian metric of signature~$(p,q')$.
\end{defn}

\begin{prop}
    Let~$T \colon M \rightarrow N$ be a smooth Cauchy time function on~$M$ with at least one compact fiber and let~$f \colon N' \rightarrow N$ be an immersion in~$N$ with~$N'$ simply connected. Then~$f^* T \colon f^* M \rightarrow N'$ is a Cauchy time function.
\end{prop}

\begin{proof}
    Let~$g \colon N_0 \rightarrow f^* M$ be an inextendible smooth causal map in~$f^* M$. Let us first show that~$f^* T \circ g \colon N_0 \rightarrow N'$ is surjective. Assume it is not and let~$(y_n)$ be a sequence in~$\Imm(f^* T \circ g)$ which converges to~$y_{\infty} \in N' \setminus \Imm(f^* T \circ g)$. Let~$(x_n)$ in~$N_0$ such that~$f^* T \circ g(x_n) = y_n$. For each~$n$, there exists an inextendible timelike map~$f_n \colon N \rightarrow M$ which goes through~$g(x_n)$. Since~$T$ is a Cauchy time function, this map also intersects~$T^{-1}(x)$ which is compact, therefore one may extract by lemma \ref{limit} a limit map~$f_{\infty}$ to the~$f_n$ which must be defined on the whole of~$N$. In particular it is defined on~$f(y_{\infty})$, which gives an inconsistency. 

    Let us now show that~$f^* T \circ g$ is a covering on~$N'$. Since it is a local diffeomorphism, we only have to show that every smooth path in~$N'$ can be uniquely lifted to a smooth path from any point in a fiber. Let~$\gamma \colon [a,b] \rightarrow N'$ be a smooth path in~$N'$. By proposition \ref{subbundleGH}, we know that~$(f \circ \gamma)^* T \colon (f \circ \gamma)^* M \rightarrow (a,b)$ is a Cauchy time function. In particular, it is isomorphic to a trivial bundle over $(a,b)$ and thus verfies the path lifting property. Thus $f^* T \circ g$ is a covering on $N'$. Since~$N'$ is assumed to be simply connected, we know that~$f^* T \circ g$ must be a diffeomorphism, hence the result.
\end{proof}

\subsection{Global hyperbolicity}

We have seen two possible generalizations of global hyperbolicity in higher signature. The first one is the existence of a Cauchy surface. This generalization is most likely not the most interesting one as it does not give sufficient informations on the global structure of the spacetime; indeed, one may find a spacetime $M$ with a Cauchy surface $S$ which remains a Cauchy surface in the space $M$ from which we removed a point $x$; take for instance $M = \R^{p,3}$, $S = \R^{p,0}$ and $x = (0, \dots, 0, 1,0,0)$. Another stronger generalization is the existence of a Cauchy time function. This on the contrary gives ample informations on the structure of the spacetime $M$ as seen previously, however it is quite difficult in practice to exhibit a Cauchy time function on a given spacetime. In the Lorentzian case the volume method used by Geroch gives ways to build Cauchy time functions on a spacetime with minimal assumptions, however those methods fail in higher signature. The correct generalization of global hyperbolicity must then lie in between. We offer a definition inspired by the compactness of causal diamond in Lorentzian signature. For $S$ an oriented embedded $(q-1)$-sphere in causal position, let $\C(S)$ be the set of causal maps $f : \D^q \rightarrow M$ up to reparametrization of $\D^q$ such that $f|_{\partial \D^q}$ is an embedding with image $S$ whose orientation given by the time orientation of $M$ corresponds to the orientation of $S$. Let us put on $\C(S)$ the topology given by the sets $\C(S)(U)$ of maps $f$ in $\C(S)$ whose image is in a given open set $U$ of $M$.

\begin{defn}
A spacetime $M$ is said to be \emph{globally hyperbolic} if it satisfies the two following conditions :

\begin{enumerate}
    \item There exists a smooth manifold $N$ such that for any inextendible causal map $f : N_0 \rightarrow M$, $N_0$ is homeomorphic to $N$.
    \item For any oriented embedded $(q-1)$-sphere $S$ in causal position in $M$, $\C(S)$ is compact.
\end{enumerate}
\end{defn}

Let us first show that this definition is implied by the existence of a Cauchy time function. The first condition is immediate as for any Cauchy time function $T : M \rightarrow N$, $T \circ f$ gives a homeomorphism between $N_0$ and $N$ for any inextendible causal map $f : N_0 \rightarrow M$. The second condition is a bit more involved. Let us start by proving the following lemma.

\begin{lemma}\label{limit}
    Let~$T \colon M \rightarrow N$ be a Cauchy time function on $M$ and let~$s_n \colon N \rightarrow M$ be a sequence of causal sections of~$T$. Assume there exists a point $x$ of $M$ such that $s_n(T(x))$ converges to $x$ in $M$. Then there exists a subsequence of~$(s_n)$ which converges locally uniformaly to a causal section~$s_{\infty} \colon N \rightarrow M$. 
\end{lemma}

\begin{proof}
    Let~$C(a,b) \subset T_x M$ be a cylindrical neighborhood centered at~$x$. Since the~$s_n(T(x))$ converges to~$x$ we know that the images of~$s_n$ are in~$C(a,b)$ for $n$ big enough. Let~$g_-$ be a constant metric on~$C(a,b)$ such that~$g_- \leqslant g$. Since every~$s_n$ is causal for~$g$, every~$s_n$ is also causal for~$g_-$ ; in particular, for $n$ big enough, every~$s_n$ can be seen as the graph of a~$1$-Lipschitz map for~$g_-$ from~$\D^q(0,b)$ to~$(-a,a)^p$, each of whom have values in a compact set. Arzela-Ascoli theorem then tells us that there exists a~$1$-Lipschitz map~$s_{\infty} \colon \D^q(0,b) \rightarrow (-a,a)^p$ to which the~$s_n$ converge uniformaly up to a subsequence. We define in this manner a map~$s_{\infty} \colon U \subset N \rightarrow M$ on an open neighborhood~$U$ of~$T(x)$. 

    \indent Let's assume that~$U$ is a maximal open set on which~$s_{\infty}$ can be defined. Assume that~$s_{\infty}$ is not inextendible. Then there must exist an inextendible path~$c \colon \R \rightarrow U$ such that~$s_{\infty} \circ c$ converges to~$y \in M$. Since the~$s_n$ converge uniformaly to~$s_{\infty}$, the sequence~$s_n(T(y))$ must converge to~$y$, and by repeating the previous construction, it becomes possible to extend~$U$ by adding an open set containing~$T(y)$. Thus the map~$s_{\infty}$ must be an inextendible causal map, and since~$T$ is a Cauchy time function,~$s_{\infty}$ must be defined on the whole of~$N$. Hence the result.
\end{proof}

\noindent Now equipped with this lemma, we can prove the following theorem.

\begin{theo}\label{compactnessofdiamonds}
    Any spacetime $M$ admitting a Cauchy time function is globally hyperbolic.
\end{theo}

\begin{proof}
    Let $S$ be an embedded oriented $(q-1)$-sphere in causal position in a spacetime $M$ admitting a Cauchy time function $T : M \rightarrow N$. Since $S$ is in causal position, there exists a causal section $s : N \rightarrow M$ for $T$ such that $T(S)$ is a hypersphere in $N$ with interior $D$ and $s(T(S)) = S$. Any element in $\C(S)$ may be glued to $s$ along $S$, therefore the elements of $\C(S)$ are exactly the causal sections $s_0 : D \rightarrow M$ and any such element can be seen as an inextendible causal map constant outside of $D$. 

    If $D$ is not a disk then $\C(S)$ is empty, therefore let us assume $D$ is a disk. Let $s_n : D \rightarrow M$ be a sequence in $\C(S)$. By using Lemma \ref{limit}, one may extract a causal section $s_{\infty} : D \rightarrow M$ to which the $s_n$ converge locally uniformaly. Since $D$ is compact, the convergence must be uniform and the set $\C(S)$ must be compact, hence the result.
\end{proof}

The order of implications between the notions we introduced are displayed in Figure \ref{fig:implications}.

\begin{figure}[h!bt]
    \begin{center}
      \includegraphics[width=.7\linewidth]{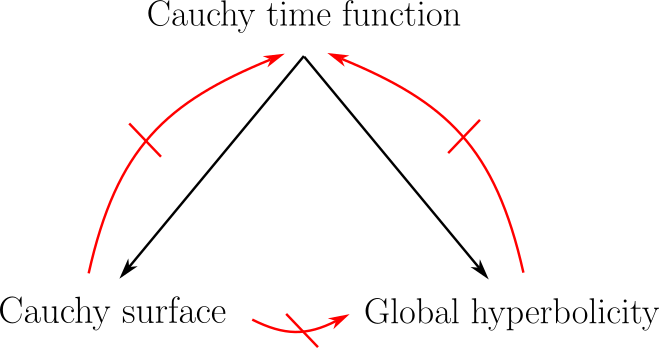}
      \end{center}
      \caption[]{The timelike curve goes from the origin to the vector~$v$.}
    \end{figure}\label{fig:implications}

We do not know as of now if there exists any globally hyperbolic space which do not admit a Cauchy surface, we suspect that none exist under some additional conditions on the spacetime $M$.

\subsection{Plateau problem in globally hyperbolic spaces}

In the Lorentzian setting, the compactness of the space of causal paths with fiexed extremities allows us to prove a central result in globally hyperbolic spaces, the existence of causal geodesics between any two causally related points (see \cite{largescalestructure}, proposition 6.7.1). We may prove an analogous result in higher signatures, by generalising the existence of a length-maximising geodesic to the solution of a Plateau problem. Let~$M$ and $S$ be as defined in the previous section. As seen in proposition \ref{diffalmosteverywhere}, any element in $\C(S)$ must be differentiable everywhere, which allows us to define an area functionnal.

\begin{defn}
    Let~$s \in \C(S)$ and let

    \[\pazocal{A}(s) = \Vol(B^{\mathrm{o}}, - s^* g_M).\]

    \noindent The function~$\pazocal{A}$ is the area functionnal on the set~$\C(S)$.
\end{defn}

\begin{theo}\label{Plateauproblem}
    There exists a~$s$ in~$C(S)$ for which~$\pazocal{A}(s)$ is maximal.
\end{theo}

It is clear that the functionnal $\pazocal{A}$ is not continuous on $\C(S)$ for the uniform topology. In the Lorentzian Minkowski space for instance, a timelike geodesic of non-zero length may be uniformaly approximated by piecewise lightlike paths which have length zero. It is however continuous on the space of $\Cc^1$ sections of $\C(S)$ for the $\Cc^1$ convergence, however this space is not compact. In order to prove the theorem, we will show that $\pazocal{A}$ is upper semi-continuous on $\C(S)$.

\begin{lemma}
The functionnal $\pazocal{A}$ is upper semi-continuous on $\C(S)$.
\end{lemma}

\begin{proof}
Let $s_0 : \overline{\D^q} \rightarrow M$ be an element in  $\C(S)$, let $x \in \D^q$ and let $C(a,b)$ be a cylindrical neighborhood centered at $s_0(x)$. Taking an arbitrarily small deformation of $g$ in $C(a,b)$ lets us assume that $s_0$ is timelike. In $C(a,b)$, $s_0$ is the graph of a map $f : \D(0,b) \rightarrow (-a,a)^p$. For $\delta > 0$ and for each $s \in \D(0, \delta)^p$, let $f_s = f_0 + s$. Since the graph of $f_0$ is timelike, there exists a $\delta > 0$ such that the graph of each $f_s$ is timelike. Let $U$ be the reunion of the graphs of the $f_s$. For each $y \in U$ at which the corresponding $f_s$ is differentiable and each $v \in T_y M$, let us write 

\[\overline{g}_y(v) = g_y(\pi_y(v)),\]   

\noindent where $\pi_y$ is the orthogonal projection on the tangent space to the graph of $f_s$. The metric $\overline{g}$ is defined almost everywhere. It is negative in $q$ directions, degenerate in $p$ directions and verifies $g \geqslant \overline{g}$. Thus for each inextendible causal map $\hat{s}$ with image in $U$, we have 

\[\pazocal{A}_{g}(\hat{s}) \leqslant \pazocal{A}_{\overline{g}}(\hat{s}).\] 

\noindent When $v$ is a vector tangent to the graph of $f_0$, $\overline{g}(v) = g(v)$. This means in particular that $\pazocal{A}_{\overline{g}}(\hat{s})$ converges to $\pazocal{A}_{g}(\hat{s})$ when $\hat{s}$ goes to $s_0$, thus for each $\varepsilon > 0$ there exists an open set $V \subset U$ containing the image of $s_0$ such that for each $\hat{s}$ with image in $V$, 

\[\pazocal{A}_{g}(\hat{s}) \leqslant \pazocal{A}_{g}(s_0) + \varepsilon.\]   

Since $\overline{\D^q}$ is compact, one only needs to do this procedure on a finite number of cylindrical neighborhood, hence the result.
\end{proof}

\begin{proof}[Proof of theorem \ref{Plateauproblem}]

Let~$M = \sup_{\C(S)} \pazocal{A}$ and let~$s_n$ be a sequence of~$\C(S)$ such that~$\pazocal{A}(s_n)$ converges to~$M$. Since~$\C(S)$ is compact,~$s_n$ converges to~$s_{\infty}$ up to an extraction and since~$\pazocal{A}$ is upper semi-continuous,~$\limsup \pazocal{A}(s_n) \leqslant \pazocal{A}(s_{\infty})$. We must then have 

\[\limsup \pazocal{A}(s_n) = M \leqslant \pazocal{A}(s_{\infty}) \leqslant M,\] 

\noindent meaning that~$\pazocal{A}(s_{\infty}) = M$, hence the result.
    
\end{proof}

\section{Global hyperbolicity in $(\so_0(p,q+1), \mathbb{H}^{p,q})$-structures}

\subsection{Causality in the Einstein space}

    As we have seen, one may define affine charts of~$\widehat{\ein}^{p,q}$ which are conformally equivalent to~$\R^{p,q}$. Since~$\widehat{\ein}^{p,q}$ and~$\R^{p,q}$ both have notions of causality which may not coincide, as inextendability depends on the spacetime, one may wonder if two points which are causally related in an affine chart of~$\widehat{\ein}^{p,q}$ are necessarily causally related in~$\widehat{\ein}^{p,q}$.
    
    \begin{prop}
    Let~$\R^{p,q} \subset \widehat{\ein}^{p,q}$ be an affine charts containing two points~$x,y$. Then~$x$ and~$y$ are causally related in~$\R^{p,q}$ if and only if they are causally related in~$\widehat{\ein}^{p,q}$.
    \end{prop}
    
    \begin{proof}
    If~$x$ and~$y$ are causally related in~$\widehat{\ein}^{p,q}$, there exists an inextendible causal map~$f \colon \s^q \rightarrow \widehat{\ein}^{p,q}$ containing both~$x$ and~$y$ in its image. The intersection of the image of this causal map with~$\R^{p,q}$ gives an causal map of~$\R^{p,q}$ containing both~$x$ and~$y$. Assume there exists a path $c : I \rightarrow f^{-1}(\R^{p,q})$ inextendible in $f^{-1}(\R^{p,q})$ such that $f \circ c$ converges in $\R^{p,q}$. The path $c$ must also be inextendible in $\s^p$ as otherwise the limit point would still be in $f^{-1}(\R^{p,q})$ which contradicts the fact that $f \colon \s^p \rightarrow \s^q$ is inextendible, thus the restriction of $f$ to $f^{-1}(\R^{p,q})$ is also inextendible. Hence~$x$ and~$y$ are causally related in~$\R^{p,q}$.
    
    If~$x$ and~$y$ are causally related in~$\R^{p,q}$, we know that the segment~$[x,y]$ must be causal from \ref{geodesics}. We may then include~$\{x,y\}$ in an affine causal space of dimension~$q$ in~$\R^{p,q}$, whose closure in~$\widehat{\ein}^{p,q}$ is a Möbius sphere of dimension $q$, hence the result.
    \end{proof}
    
    \noindent Since~$\s^p$ and~$\R^p$ are both complete with respect to their standard metric, proposition \ref{mainexample} tells us that both~$\widehat{\ein}^{p,q} \simeq \s^p \times \s^q$ and~$\R^{p,q}$ are globally hyperbolic with the projection on the second factor as a Cauchy time function. This means in particular that the inextendible causal maps of~$\widehat{\ein}^{p,q}$ (resp.~$\R^{p,q}$) are the graphs of~$1$-Lipschitz maps~$s \colon \s^q \rightarrow \s^p$ (resp.~$s \colon \R^q \rightarrow \R^p$).
    
    For a point $x$ in~$\widehat{\ein}^{p,q}$, the set of points causally related to~$x$ is a space which, in a sense, is too big to be interesting. In order to get a smaller space, one needs to consider the causal diamonds.
    
    \begin{defn}
    A set~$E$ in~$\widehat{\ein}^{p,q}$ is said to be \emph{proper} if its closure is contained within an affine chart of~$\widehat{\ein}^{p,q}$. An equivalent definition when $E$ is connected is that there exists an isotropic cone in~$\widehat{\ein}^{p,q}$ which does not meet the closure of~$E$.
    \end{defn}

    \noindent Let~$\C$ be a causal (in the sense of definition \ref{causalset}) oriented closed submanifold of dimension~$q-1$ in~$\widehat{\ein}^{p,q}$. The set $\C$ is the graph of a~$1$-Lipschitz map from~$\pi(\C) \subset \s^q$ to~$\s^p$ which may be extended to a~$1$-Lipschitz map on~$\s^q$. As~$\pi(C)$ is an oriented closed submanifold of codimension~$1$ in~$\s^q$, it must be the boundary of a compact manifold with boundary~$B$ determined by the orientation of $\C$ and the time orientation of $\widehat{\ein}^{p,q}$ with $B$ as the interior of $\C$. We will denote by $\Diam(\C)$ the open future of $\C$ in $\widehat{\ein}^{p,q}$ as defined in definition \ref{openfuture}.
    
    \begin{defn}
    The submanifold~$\C$ is said to be \emph{purely lightlike} if~$\Diam(\C)$ is empty for at least one orientation of $\C$.
    \end{defn}
    
    \noindent From now on we will asume that~$\C$ is not purely lightlike.
    
    \begin{prop}\label{properdiamond}
    The set~$\Diam(\C)$ is proper in~$\widehat{\ein}^{p,q}$.
    \end{prop}
    
    \begin{proof}
    Let~$\overline{\C}$ be the submanifold~$\C$ with opposite orientation and let~$x = (x_1,x_2) \in \Diam(\overline{\C}) \subset \s^p \times \s^q$. Let us assume that the isotropic cone of~$x$ meets the closure of~$\Diam(\C)$ in a point~$(y_1,y_2)$. Then there must exist a lightlike geodesic of~$\widehat{\ein}^{p,q}$ going from~$(x_1, x_2)$ to~$(y_1, y_2)$, which tells us in particular that~$d_{\s^p}(x_1, y_1) = d_{\s^q}(x_2, y_2)$. Since~$(x_1, x_2)$ and~$(y_1, y_2)$ are both in the closure of the diamonds of~$\C$ for the two orientations, there must exist a map~$s \colon \s^q \rightarrow \s^p$ containing both~$(x_1, x_2), (y_1, y_2)$ and~$\C$ in its graph. The projection~$\pi(\C)$ of~$\C$ in~$\s^q$ splits~$\s^q$ in two connected components, one containing~$x_2$ and the other containing~$y_2$. Let~$c$ be a geodesic on~$\s^q$ going from~$x_2$ to~$y_2$ in~$\s^q$ and let~$z$ be a point in~$c$ which is also in~$\pi(\C)$. The restriction of~$s$ to~$c$ is a~$1$-Lipschitz map containing such that~$s(x_2) = x_1$,~$s(y_2) = y_1$ with~$d_{\s^p}(x_1, y_1) = d_{\s^q}(x_2, y_2)$. This implies that~$s$ is an isometry, meaning that the point~$(s(z),z)$ is on the isotropic cone of~$(x_1, x_2)$. This is not possible as~$(x_1, x_2)$ was assumed to be in the causal diamond of~$\overline{\C}$, hence the result.
    \end{proof}
    
    \begin{coro}
    In particular, every closed causal submanifold of dimension~$q-1$ in~$\widehat{\ein}^{p,q}$ which is not purely lightlike must be proper.
    \end{coro}
    
    \begin{rem}
    It is also possible to reverse the sign and build a proper space using a closed submanifold of dimension~$p-1$ which is causal (resp. timelike) in~$\widehat{\ein}^{q,p} = - \widehat{\ein}^{p,q}$. Such a set will be called \emph{achronal} (resp. \emph{acausal}).
    \end{rem}
    
    \noindent Since every causal diamond is proper, we may study them in the space~$\R^{p,q}$. The most regular type of causal diamond is the flat causal diamond, obtained by taking a timelike Möbius sphere in~$\widehat{\ein}^{p,q}$. As we have seen in Proposition \ref{causaldiamonds}, the diamond is then conformally equivalent to~$\h^p \times \h^q$; in particular, it is globally hyperbolic.

    \subsection{Positive causal diamonds admit Cauchy time functions}
    
    One may wonder if diamonds associated to less symmetric causal submanifolds of dimension $q-1$ may admit a Cauchy time function. One may check that~$\Diam(\C)$ almost never admits a Cauchy time function when~$\C$ is a causal~$(q-1)$-sphere. We need to consider a reversed diamond $\Diam^+(\C)$ obtained from an achronal closed submanifold $\C$ of dimension~$p-1$.
    
    \begin{lemma}
    Let $B$ be a compact submanifold of $\R^p$ of dimension $p$ with boundary. Let $s : \partial B \rightarrow \R^q$ be a $1$-Lipschitz map whose graph is non-purely lightlike and let $s_0 : B \rightarrow \R^q$ be a $1$-contracting map on $B^{\circ}$ equal to $s$ on $\partial B$. There exists a family~$(s_i)_{i \in I}$ of~$1$-contracting maps extending~$s$ containing~$s_0$ and whose graphs form a foliation of~$\Diam^+(s)$.
    \end{lemma}
    
    \begin{proof}
    Let~$C$ be the set of~$1$ contracting maps~$\hat{s} \colon B \rightarrow \R^q$ extending~$s$ and let~$v$ be a unit vector in~$\R^q$. Let~$C_v$ be the subset of~$C$ made of the maps~$\hat{s}$ such that for each~$x \in B$,~$\hat{s}(x) \in s_0(x) + \R v$. In particular, for each~$v$ unit vector in $\R^q$,~$s_0$ is in $C_v$. \\
    
    \noindent The set~$C_v$ admits two extremal points in its closure for the pointwise topology which graph are the upper and lower boundaries of the intersection of $\Diam^+(s)$ with the set $\{s_0(x) + \R v, x \in B\}$. Let~$s_v$ be an extremal point verifying that for all~$x$,~$s_v(x) \in s_0(x) + \R^{>0} v$. Since~$C$ and~$C_v$ are both convex, the family~$\{ts_v + (1-t)s_0\}_{t \in (0,1]}$ belongs to~$C_v$ and by construction and the family~$\{t s_v + (1-t)s_0\}_{v \in \s^{q-1}, t \in (0, 1]}$ contains~$s_0$. Since for each $x$ and $v$, the graph of $s_v$ is the upper boundary of the intersection of $\Diam^+(s)$ with the set $\{s_0(x) + \R v, x \in B\}$, the graphs of the segment $[s_v, s_0]$ must foliate the intersection between $\Diam^+(s)$ and $\{s_0(x) + \R^{>0} v, x \in B\}$ which gives us the result.
    \end{proof}
    
    \begin{lemma}\label{surface into foliation}
    There exists a map~$T \colon \Diam^+(s) \rightarrow \D^q$ for which the level sets are the graphs of the $(s_i)$.
    \end{lemma}
    
    \begin{proof}
    Let~$y \in B^\circ$ and~$D = \Diam^+(s) \cap (y + \R^{0,q})$. For each~$x \in \Diam^+(s)$, there exists a unique~$s_i$ containing~$x$ in its graph. This graph intersects~$D$ in a unique point by construction which we will denote by~$T(x)$. The map $T$ is continuous since the family $(s_i)_I$ varies continuously with $i$. By construction, the level sets of~$T$ in~$\Diam^+(s)$ are the graphs of the~$s_i$. For each~$x$ in the graph of~$s$, the intersection of the positive part of the isotropic cone with~$y+\R^{0,q}$ is convex. Since~$D$ is the intersection of each of those convex spaces for~$x$ in the graph of~$s$,~$D$ must be convex and thus homeomorphic to~$\D^q$, hence the result.
    \end{proof}

    \begin{prop}\label{diamondGH}
        Let $T, I, (s_i)_{i \in I}$ be as in the previous lemmas, then the map~$T$ is a Cauchy time function.
    \end{prop}

    \begin{proof}
    Let~$f \colon N_0 \rightarrow \Diam^+(s)$ be an inextendible causal map. We will start by showing that~$T \circ f \colon N_0 \rightarrow \D^q$ is a local homeomorphism. Let~$x \in N_0$. Since~$f$ is causal, there exists an open set~$U$ containing~$x$ such that in~$U$,~$f$ is the graph of a~$1$-Lipschitz map from an open set of~$\R^q$ to~$\R^p$. By Kirszbraun's Theorem this map may be extended to a~$1$-Lipschitz map~$g \colon \R^q \rightarrow \R^p$. Each leaf~$(T=cste)$ is the graph of a~$1$-contracting map which may be extended to a map outside of~$B$,~$\hat{s_i} \colon \R^p \rightarrow \R^q$. The composition~$g \circ \hat{s_i} \colon \R^p \rightarrow \R^p$ is~$1$-contracting and thus admits at most one fixed point which corresponds to intersection between the graphs of~$g$ and~$s_i$. Since the level sets of $T$ foliate a neighborhood of $f(x)$, we know that those intersections exist in a neighborhood of $x$, meaning that the map $T \circ f$ is a local homeomorpism.
    
    We will next show that~$T \circ f$ is a covering onto~$\D^q$. For this, we agin prove a pth lifting property. Let~$c \colon [0,1] \rightarrow \D^q$ be a path starting in the image of~$T \circ f$. Since~$T \circ f$ is a local homeomorphism, one may pull it back into a path~$\tilde{c} \colon [0,a) \rightarrow N_0$ such that~$T \circ f \circ \tilde{c} = c|_{[0,a)}$. Let us show that~$\tilde{c}$ can be extended into the whole of~$[0,1]$. Since~$c$ converges in~$\D^q$ and~$f \circ \tilde{c}$ is causal,~$f \circ \tilde{c}$ must converge in~$a$ to a point in~$\overline{T^{-1}(a)}$. Assume that this limit~$\ell$ is in~$\partial T^{-1}(a) = graph(s)$. By construction of~$\Diam^+(s)$, each point~$x$ in~$\Diam^+(s)$ must verify that~$[x,\ell]$ is spacelike. In particular, for each~$t \in [0,a)$,~$[f(\tilde{c}(t)), \ell]$ is spacelike. This is not possible near~$a$ since~$f \circ \tilde{c}$ is causal and converges to~$\ell$, thus~$\ell$ must be in~$\Diam^+(s)$. Since~$f$ was taken to be inextendible in~$\Diam^+(s)$,~$\tilde{c}$ must then converge in~$N_0$ and can then be extended further than~$a$, hence the result. Since this can be done for any path~$c$,~$T \circ f$ must be a covering on~$\D^q$. Since~$\D^q$ is simply connected,~$T \circ f$ is a homeomorphism from~$N_0$ to~$\D^q$, hence the result.
    \end{proof}
    
    \begin{coro}
    Let~$s_0 \colon B \rightarrow \R^q$ be a~$1$-contracting map extending~$s$. The graph of~$s_0$ is a Cauchy surface in~$\Diam^+(s)$.
    \end{coro}
    
    \begin{proof}
    Lemma \ref{surface into foliation} tells us that~$s_0$ may be completed into a foliation defined by~$T \colon \Diam^+(s) \rightarrow \D^q$ and Theorem \ref{diamondGH} that~$T$ must then be a Cauchy time function. In particular, the graph of~$s_0$ must be a Cauchy surface as it is a level set of $T$.
    \end{proof}
    
    \begin{rem}
    This result gives us examples of spaces admitting a Cauchy time function which do not have a simple conformal model.
    \end{rem}
    
    \begin{prop}\label{causallyconvex}
    Let~$x,y$ two points in~$\Diam^+(s)$ which are causally related. Then the segment~$[x,y]$ is causal and the two points~$x,y$ are causally related in~$\R^{p,q}$ (and thus in~$\widehat{\ein}^{p,q}$).
    \end{prop}
    
    \begin{proof}
    Assume that~$x,y$ are causally related in~$\Diam^+(s)$ and that~$[x,y]$ is spacelike. Then there must exist a~$1$-contracting map~$s_0 \colon B \rightarrow \R^q$ extending~$s$ containing both~$x$ and~$y$ in its graph. By the previous corollary, any inextendible causal map only intersects the graph of~$s_0$ in one point, meaning that~$x$ and~$y$ cannot be causally related in~$\Diam^+(s)$, hence the result.
    \end{proof}

\begin{rem}
To my knowledge, there does not exist a more direct proof of this fact as we have very little control a priori over which inextendible maps in $\Diam^+(s)$ may be extended into an inextendible map in $\R^{p,q}$.
\end{rem}

\subsection{GH-regular representations in $\so_0(p, q+1)$}

Let~$\R^{p, q+1}$ be the vector space~$\R^{p+q+1}$ endowed with the quadratic form

\[Q \colon (x_1,\cdots,x_p,y_1,\cdots,y_{q+1}) \longmapsto x_1^2 + \cdots + x_p^2 - y_1^2 - \cdots - y_{q+1}^2.\]

\noindent Let~$\h^{p, q}$ be the hypersurface $\{Q = -1\}$ of~$\R^{p,q+1}$. As we have seen in section \ref{pseudoriemannianexamples}, since the tangent space of~$\h^{p,q}$ at~$x \in \h^{p,q}$ is identified with~$x^{\perp}$,~$Q$ endows~$\h^{p,q}$ with a pseudo-Riemannian metric of signature~$(p,q)$. We sometimes prefere to use the spherical coordinates on the negatives factor $\R^{p, q+1} \simeq \R^p \times \R^+ \times \s^{q}$, which we will write $(x_1,\cdots, x_p, y_1,\cdots, y_{q+1}) \simeq (x_1,\cdots,x_p, r, s)$.

\begin{prop}[\cite{beyrer2023mathbbhpqconvex}, prop. 2.3]
The space~$\h^{p,q}$ is conformally equivalent to~$(\D^p \times \s^q, g_{\s^p} - g_{\s^q})$ where~$\D^p$ is seen as the upper hemisphere of~$\s^p$ via the map 

\[\psi \colon \left( x_1, \cdots, x_p, r, s\right) \in \R^{p,0} \times \R^+ \times \s^q \longmapsto \left( \frac{1}{r}, \frac{x_1}{r}, \cdots, \frac{x_p}{r}, s \right) \in \R^p \times \s^q.\]

\noindent In particular,~$\h^{p,q}$ can be conformally embedded in~$\widehat{\ein}^{p,q}$ with boundary equal to~$\widehat{\ein}^{p-1,q}$.
\end{prop}

\begin{rem}
In the projective model of~$\h^{p,q}$ in~$\mathbb{P}(\R^{p, q+1})$, the boundary~$\partial \h^{p,q}$ is identified with the projectivisation of the isotropic cone~$\mathcal{C}(Q)$ of~$Q$.
\end{rem}

\noindent Let~$\Lambda$ be an achronal~$(p-1)$-sphere in~$\partial \h^{p,q}$ and let~$\Omega(\Lambda)$ be the positive diamond spanned by~$\Lambda$ in~$\widehat{\ein}^{p,q}$ which is included in $\h^{p,q}$. This does not depend on the orientation of $\Lambda$. From now on, we will assume that~$\Lambda$ is not purely lightlike, i.e that~$\Omega(\Lambda)$ is non-empty.

\begin{defn}
Let~$\rho \colon \Gamma \rightarrow \so_0(p, q+1)$ be a discrete and faithful representation from a finitely generated, torsion free group~$\Gamma$. We'll say that $\rho$ is \emph{GH-regular} if there exists an achronal non-purely lightlike $(p-1)$-sphere in $\partial \h^{p,q}$ which is stabilized by $\rho$.
\end{defn}

This terminology is inspired by the works of Barbot \cite{barbot2013deformations}. When $\rho$ is GH-regular, the action of~$\rho$ on~$\h^{p,q}$ must stabilize~$\Omega(\Lambda)$.

\begin{prop}
The action of~$\rho$ on~$\Omega(\Lambda)$ is properly discontinuous.
\end{prop}

\begin{proof}
We have proven in Proposition \ref{properdiamond} than causal diamonds are proper in the Einstein space. In particular,~$\Omega(\Lambda)$ must be proper in~$\widehat{\ein}^{p,q}$, which implies that the action of~$\rho$ on~$\Omega(\Lambda)$ must also be proper. One may see this by showing that~$\rho$ acts isometrically on the convex~$\Omega(\Lambda)$ for the Hilbert metric or by applying corollary 5.3 of \cite{Zimmer2}. 
\end{proof}

\noindent Let~$\Omega(\Lambda)/\rho$ be the quotient of~$\Omega(\Lambda)$ by~$\rho$. It is a pseudo-Riemannian manifold locally isometric to~$\h^{p,q}$. We know from Theorem \ref{diamondGH} that~$\Omega(\Lambda)$ admits a Cauchy time function, however since the foliation defined by this function is not necessarily invariant by~$\rho$ we do not necessarily get a Cauchy time function on~$\Omega(\Lambda)/\rho$. In fact, we will see that such a function does not exist in general. It is however possible to exhibit a Cauchy surface of~$\Omega(\Lambda)$ which is stable by~$\rho$.

\begin{defn}
Let~$S$ be a smooth complete spacelike submanifold of~$\h^{p,q}$ of dimension~$p$. We will say that~$S$ is \emph{maximal} if it has zero mean curvature at each point.
\end{defn}

\noindent We know from the work of Seppi-Smith-Toulisse \cite{seppi2023completemaximalsubmanifoldspseudohyperbolic} that under our assumption, there must exist a unique maximal submanifold with boundary~$\Lambda$.

\begin{theo}[\cite{seppi2023completemaximalsubmanifoldspseudohyperbolic}, Theorem A]
Let~$\Lambda$ be an achronal non-purely lightlike~$(p-1)$-sphere in~$\partial \h^{p,q}$. Then there exists a unique maximal submanifold~$S(\Lambda)$ of~$\h^{p,q}$ with boundary~$\Lambda$.
\end{theo}

\noindent We will only be interested in the fact that since~$S(\Lambda)$ is unique, it must in particular be preserved by~$\rho$. Let us note that the proof of Theorem A of \cite{seppi2023completemaximalsubmanifoldspseudohyperbolic} is very involved, and that while it is believed that one should be able to build a spacelike submanifold stable by~$\rho$ in a purely geometric fashion, it is currently not known how. By corollary of Theorem \ref{diamondGH} we know that~$S(\Lambda)$ must be a Cauchy surface. Let us show that the quotient~$S(\Lambda)/\rho$ must then be a Cauchy surface of~$\Omega(\Lambda)/\rho$.

\begin{prop}\label{Cauchyquotient}
Let~$\Gamma$ be a discrete group acting properly conformally on a pseudo-Riemannian conformal space~$M$ and let us assume that~$M$ admits a Cauchy surface~$S$ stable by~$\Gamma$. Then~$S/\Gamma$ is a Cauchy surface of~$M/\Gamma$.
\end{prop}

\begin{proof}
Let~$\pi_1 \colon M \rightarrow M/\Gamma$ be the projection, let~$f \colon N_0 \rightarrow M/\Gamma$ be an inextendible causal map and let~$\pi_2 \colon \tilde{N_0} \rightarrow N_0$ be the universal covering of~$N_0$ with~$\tilde{f} \colon \tilde{N_0} \rightarrow M/\Gamma$ continuous such that~$\tilde{f} = f \circ \pi_2$. Let~$\hat{f} \colon \tilde{N_0} \rightarrow M$ verifying that~$\pi_1 \circ \hat{f} = \tilde{f}$. The map~$\hat{f}$ is also an inextendible causal map as any path converging in $M$ converges in $M/\Gamma$, therefore there exists a unique~$x \in \tilde{N_0}$ such that~$\hat{f}(x) \in S$, meaning that~$f(\pi_2(x)) \in S/\rho$. Since~$x$ is unique, this proves the result.
\end{proof}

\noindent In particular, this result tells us that~$\Omega(\Lambda)/\rho$ admits a Cauchy surface~$S(\Lambda)/\rho$ which is complete as a Riemannian manifold since~$S(\Lambda)$ is complete in~$\h^{p,q}$.

\begin{comment}

\begin{prop}\label{cocompactaction}
Let~$\rho \colon \Gamma \rightarrow \so(p,q+1)$ be a~$\h^{p,q}$-convex cocompact representation of a Gromov hyperbolic group~$\Gamma$ with boundary a spacelike~$(p-1)$-sphere in~$\widehat{\ein}^{p-1,q}$. Then every Cauchy surface of~$\Omega(\Lambda)/\rho$ is compact. 
\end{prop}

\begin{proof}
Since~$\rho$ is Anosov, the Gromov boundary of~$\Gamma$ must be homeomorphic to a~$(p-1)$-sphere by proposition \ref{DGKequivalence}, which in turns implies that the cohomological dimension of~$\Gamma$ must be equal to~$p$. Since furthermore~$\rho$ acts properly discontinuously on the contractible manifold~$S(\Lambda)$ of dimension~$p$, the action must be cocompact meaning that~$S(\Lambda)/\rho$ is compact. Finally, since~$\Omega(\Lambda)/\rho$ admits a foliation by timelike submanifolds of dimension~$q$, Proposition \ref{timelikefoliation} tells us that all Cauchy surfaces must be homeomorphic to each other, meaning that every Cauchy surface of~$\Omega(\Lambda)/\rho$ is compact.
\end{proof}

\end{comment}

\begin{prop}
The space~$\Omega(\Lambda)/\rho$ is globally hyperbolic.
\end{prop}

\begin{proof}
Let $f : N_0 \rightarrow \Omega(\Lambda)/\rho$ be an inextendible causal map. We must show that $N_0$ is homeomorphic to $\D^q$. Since $\Omega(\Lambda)/\rho$ admits a Cauchy surface $S(\Lambda)/\rho$, the space $\Omega(\Lambda)/\rho$ must be causal and $N_0$ must be simply connected. Let $\tilde{f} : N_0 \rightarrow \Omega(\Lambda)$ be such that $\pi \circ \tilde{f} = f$ for $\pi : \Omega(\Lambda) \rightarrow \Omega(\Lambda)/\rho$ the projection. The map $\tilde{f}$ is an inextendible causal map in $\Omega(\Lambda)$ which admits a Cauchy time function with values in $\D^q$, therefore $N_0$ must be homeomorphic to $\D^q$, hence the result.

Let us now show that $\Omega(\Lambda)/\rho$ has compact diamonds. Let us first consider the case $q \geq 3$. Let $f : \s^{q-1} \rightarrow \Omega(\Lambda)$ be a map with image in causal position. Since $q \geq 3$, $\s^{q-1}$ is simply connected. Let $\tilde{f} : \s^{q-1} \rightarrow \Omega(\Lambda)$ be a lift of $f$. Since $\Omega(\Lambda)$ admits a Cauchy time function, it is globally hyperbolic and the space $\C(Im(\tilde{f}))$ is compact, hence the result.

Assume now $q=2$. Let $f : \s^{1} \rightarrow \Omega(\Lambda)$ be a map with image in causal position and let $\hat{f} : \R \rightarrow \Omega(\Lambda)/\rho$ be its associated periodic map. Let $\tilde{f}$ be a lift of $\hat{f}$ to $\Omega(\Lambda)$ through a point $x \in \Omega(\Lambda)$. If $\tilde{f}$ is periodic, it can be seen as a map from $\s^1$ to $\Omega(\Lambda)$ in causal position which is enough to conclude since $\Omega(\Lambda)$ is globally hyperbolic. If $\tilde{f}$ is not periodic, then it must intersect the orbit of $x$ in an infinite amount of points. Therefore there exists a sequence $\gamma_n \in \Gamma$ such that for all $n$, the points $x$ and $\rho(\gamma_n) x$ are in causal position. Since the action of $\rho$ on $\Omega(\Lambda)$ is proper, the sequence $\rho(\gamma_n) x$ must have an accumulation point in the boundary of $\Omega(\Lambda)$ in $\widehat{\ein}^{p,2}$; this accumulation point $\ell$ must be in $\Lambda$, for example because the Hilbert distance in $\Omega(\Lambda) \subset \mathbb{P}(\R^{p, 3})$ from $\rho(\gamma_n) x$ to $S(\Lambda)$ must remain constant. Since for each $n$ the segment $[x, \rho(\gamma_n) x]$ is causal, the segment $[x, \ell]$ must also be causal which is impossible as $\ell$ is in $\Lambda$ and $x$ in $\Omega(\Lambda)$. This gives the result for $q=2$.

The result is well known in the Lorentzian case $q=1$, see for instance \cite{barbot2013deformations}, \cite{smai22}.
\end{proof}

\begin{rem}
This gives an example of a pseudo-Riemannian spacetime which is globally hyperbolic but does not admit a Cauchy time function. As shown in \cite{CollierTholozanToulisse}, when~$p=2$ and~$\rho$ is a maximal representation for a surface group~$\Gamma$,~$\Omega(\Lambda)/\rho$ is a fiber bundle over~$S/\rho$ given by a foliation by totally geodesic timelike disks. If~$\Omega(\Lambda)/\rho$ was to admit a Cauchy time function, it would trivialize this bundle; however we know that this bundle may not be trivial for some~$\rho$.
\end{rem}

\subsection{GH-regular representations as holonomies of $(\so_0(p,q+1), \mathbb{H}^{p,q})$-structures}

We have now proven that any GH-regular representation $\rho$ is the holonomy of a $(\so_0(p,q+1), \h^{p,q})$-structure $\Omega(\Lambda)/\rho$ which is globally hyperbolic and admits a complete Cauchy surface. Let us note another important property of those structures :

\begin{defn}
Let $M$ be a spacetime of signature $(p,q)$. We'll say that $M$ is time-convex if for any totally geodesic inextendible timelike immersion $f : N_0 \rightarrow M$ and for any two points $x,y$ in $N_0$, there exists a unique geodesic between $x$ and $y$ in $N_0$ for the pullback metric of $M$ by $f$.
\end{defn}

\begin{prop}
Let $\rho$ be a GH-regular representation. The space $\Omega(\Lambda)/\rho$ is time-convex.
\end{prop}

\begin{proof}
Let $f : \D^q \rightarrow \Omega(\Lambda)/\rho$ be a totally geodesic inextendible timelike map. Let $\tilde{f}$ be a lift of $f$ to $\Omega(\Lambda)$. Since $\Omega(\Lambda)$ is a proper projective convex set in $\mathbb{P}(\R^{p, q+1})$, there exists a unique geodesic in $\Omega(\Lambda)$ between any two points in the image of $\tilde{f}$. This geodesic is timelike and completely contained in the image of $\tilde{f}$. This yields the result.
\end{proof}

The GH-regular representations $\rho : \Gamma \rightarrow \so_0(p, q+1)$ are therefore holonomies of $(\so_0(p, q+1), \h^{p,q})$-structures which are globally hyperbolic, time-convex with a complete Cauchy surface. For $q=1$, it was shown by Mess in \cite{mess2007lorentz} that those properties are enough to characterize GH-regular representations. We are going to show that this is still true for $q \geq 2$.

\begin{theo}
    Let $M$ be a $(\so_0(p, q+1), \h^{p,q})$-structure which is globally hyperbolic, time convex and with a complete Cauchy surface. Then there exists a GH-regular representation $\rho : \pi_1(M) \rightarrow \so_0(p,q+1)$ such that $M$ is isometrically embedded in $\Omega(\Lambda)/\rho$.
\end{theo}

\begin{proof}
Let $M$ be as stated and let $D : \widetilde{M} \rightarrow \h^{p,q}$ be the developping map of $M$. Let $S$ be a complete Cauchy surface of $M$. The lift $\widetilde{S}$ of $S$ in $\widetilde{M}$ must also be a complete Cauchy surface of $\widetilde{M}$. Indeed, let $f : N_0 \rightarrow \widetilde{M}$ be an inextendible causal map. The projection $\pi \circ f$ remains an inextendible causal map in $\Omega(\Lambda)/\rho$, thus there exists a unique $x \in N_0$ such that $\pi \circ f(x) \in S$. This implies that $x$ is the unique point in $N_0$ such that $f(x) \in \widetilde{S}$, therefore $\widetilde{S}$ is a Cauchy surface in $\widetilde{M}$.

We will now show that $\widetilde{M}$ remains time convex. Let $f : N_0 \rightarrow \widetilde{M}$ be a totally geodesic inextendible timelike map, the projection $\pi \circ f : N_0 \rightarrow M$ remains a totally geodesic inextendible timelike map as $\pi$ is locally isometric. Thus for any two points $x,y$ in $N_0$, there exists a unique geodesic between $x$ and $y$ for the metric $f^* (\pi^* g_M)$. Since $\pi$ is locally isometric, $\pi^* g_M = g_{\widetilde{M}}$, which yields the result. 

Let us show that $D : \widetilde{S} \rightarrow \h^{p,q}$ is an embedding. Since $\widetilde{S}$ is complete, $D|_{\widetilde{S}}$ must be an inextendible spacelike map in $\h^{p,q}$. Since $\h^{p,q}$ is conformally equivalent to $\D^{p} \times \s^q$ where $\s^q$ is compact (thus complete) and $\D^p$ is simply connected, the projection on $\D^p$ is a Cauchy time function in $- \h^{p,q}$ and $D|_{\widetilde{S}}$ is the graph of a $1$-contracting map from $\D^p$ to $\s^q$. In particular, $\widetilde{S}$ is homeomorphic to $\D^p$ and $D|_{\widetilde{S}}$ is an embedding.

The boundary $\Lambda$ of $D(\widetilde{S})$ in $\h^{p,q}$ must then be an achronal $(p-1)$-sphere $\Lambda$ in $\partial \h^{p,q}$ which must also be non-purely lightlike since $\Omega(\Lambda)$ contains $D(\widetilde{S})$. Let $\rho : \pi_1(M) \rightarrow \so_0(p,q)$ be the action of the holonomy of $\widetilde{M}$ on $\h^{p,q}$ via $D$. Then $\Lambda$ must also be stable by $\rho$, the action of $\rho$ on $\Omega(\Lambda)$ must be proper and $\Omega(\Lambda)/\rho$ verifies all the previous properties. All that is left is to show that $D$ embeds $\widetilde{M}$ in $\Omega(\Lambda)$.

First, let's show that the image of $D$ is entirely contained in $\Omega(\Lambda)$. Assume this is not the case and take $x \in Im(D) \setminus \Omega(\Lambda)$. Since $x$ is not in $\Omega(\Lambda)$, there exists $y$ in $\Lambda$ such that $x$ and $y$ are causally related in $\widehat{\ein}^{p,q}$. Let us consider an inextendible map $f$ in $\widehat{\ein}^{p,q}$ containing both $x$ and $y$. The image of this map can never be in $\Omega(\Lambda)$ as any point in the intersection would be causally related to $y$; in particular, this map can never intersect $D(\widetilde{S})$. The preimage of this map by $D$ is non-empty as $x$ was assumed to be in the image of $D$, therefore it is an inextendible causal map of $\widetilde{M}$ which never meets $\widetilde{S}$. This is impossible as $\widetilde{S}$ is a Cauchy surface of $\widetilde{M}$. Therefore the image of $D$ must lie in $\Omega(\Lambda)$.

Let $\pazocal{F}$ be a foliation by totally geodesic timelike subspaces of dimension $q$ of $\widetilde{M}$. Such a foliation exists, for instance by pulling back one on $\Omega(\Lambda)$. We know that each of those leaves must be convex for the induced metric by $\widetilde{M}$, meaning that for each $x \in \widetilde{M}$, there exists a unique leaf $\pazocal{F}_x$ containing $x$, a unique $\pi_x$ both in $\widetilde{S}$ and in $\pazocal{F}_x$ as $\widetilde{S}$ is a Cauchy surface and a unique geodesic from $\pi_x$ to $x$ in $\pazocal{F}_x$. This gives an embedding of each leaf in $\Omega(\Lambda)$ as the set $\Omega(\Lambda)$ is also time convex, therefore $D$ is a $\rho$-equivariant embedding of $\widetilde{M}$ in $\Omega(\Lambda)$ and $M$ can be embedded within $\Omega(\Lambda)/\rho$. This ends the proof.
\end{proof}

\begin{rem}
In particular, the GH-regular representations are exactly the holonomies of \emph{maximal} globally hyperbolic $(\so_0(p, q+1), \h^{p,q})$-structures with complete Cauchy surface which are time convex.
\end{rem}

In the case $q=1$, Mess proves the result without the assumption that the timelike spaces are convex; this is because this property is implied by the fact that the space is globally hyperbolic. When $q > 1$ this condition is necessary. Let us exhibit a counter-example without this assumption for $q = 2$. Let $\Lambda$ be a positive Möbius $(p-1)$-sphere in $\partial \h^{p,2}$. Let $S(\Lambda)$ be the totally geodesic copy of $\h^p$ in $\h^{p,2}$ with boundary $\Lambda$ (in particular, it is the same $S(\Lambda)$ as given by Seppi-Smith-Toulisse). The space $\Omega(\Lambda) \setminus S(\Lambda)$ is not simply connected as its fundamental group is equal to $\mathbb{Z}$. Let $M$ be the universal cover of $\Omega(\Lambda) \setminus D(\Lambda)$.

\begin{prop}
The space $M = \widetilde{\Omega(\Lambda) \setminus D(\Lambda)}$ is locally isometric to $\h^{p,2}$, is globally hyperbolic and admits a complete Cauchy surface; however, it does not embed in $\Omega(\Lambda)$.
\end{prop}

\begin{proof}
It is clear that $M$ is locally isometric to $\h^{p,2}$. Furthermore, since $\Lambda$ is a Möbius sphere, $\Omega(\Lambda)$ is conformally equivalent to $\h^p \times \h^2$ and $\Omega(\Lambda) \setminus D(\Lambda)$ is conformally equivalent to $\h^p \times (\h^2 \setminus \{x\})$ where $x$ is the center of $\h^2$. Finally, $M$ if conformally equivalent to $\h^p \times \widetilde{(\h^2 \setminus \{x\})}$. Since $\h^p$ is complete and $\widetilde{\h^2 \setminus \{x\}}$ is simpky connected, the projection on $\widetilde{\h^2 \setminus \{x\}}$ is a Cauchy time function. In particular the space $M$ is globally hyperbolic and any subset $S = \h^p \times \star$ is a Cauchy surface. The projection of $S$ in $\Omega(\Lambda) \setminus D(\Lambda)$ is a positive submanifold with boundary $\Lambda$, therefore it is complete with the metric induced by $\h^{p,2}$. If $M$ was to be embedded in a space of the $\Omega(\Lambda')/\Gamma$, we would have $\Gamma=\{Id\}$ since $M$ is simply connected and $\Lambda' = \Lambda$ since the projection of $S$ has $\Lambda$ as boundary. However $M$ cannot be embedded into $\Omega(\Lambda)$, for instance because the subsets $\star \times \widetilde{\h^2 \setminus \{x\}}$ are totally geodesic timelike submanifolds of $M$ with infinite area, something which does not exist in $\Omega(\Lambda)$ where totally geodesic timelike submanifolds have area at most the area of a $q$-sphere of radius $1$.
\end{proof}

Let us finish by discussing the case where $\Omega(\Lambda)/\rho$ admits a Cauchy surface which is not only complete but also compact. In this case, $\rho$ must act properly and cocompactly on a Cauchy surface of $\Omega(\Lambda)$ which are all contractible of dimension $p$, therefore $\Gamma$ must have cohomological dimension equal to $p$. Since we have used the GH-regular terminology from \cite{barbot2013deformations}, the correct name for those representations would be \emph{GHC-regular} (Global Hyperbolic Compact), however those representations have already been extensively studied in many works including Danciger-Guéritaud-Kassel in \cite{danciger2017convex}, \cite{DGK17} and Beyrer-Kassel in \cite{beyrer2023mathbbhpqconvex}. In those works, GH-regular representations with compact Cauchy surfaces in $\Omega(\Lambda)/\rho$ are called \emph{spacelike cocompact}; they are of particular interests as it was shown by Beyrer-Kassel in \cite{beyrer2023mathbbhpqconvex} that the set of spacelike cocompact representations in the character variety of $\Gamma$ in $\so_0(p, q+1)$ is a union of connected components. In the special case where $\Gamma$ is a hyperbolic group, it was shown that the sphere $\Lambda$ must be acausal (Beyrer-Kassel, \cite{beyrer2023mathbbhpqconvex}) and that the representation $\rho$ must be $P_1$-Anosov in $\so_0(p,q+1)$ (Danciger-Guéritaud-Kassel, \cite{danciger2017convex}, \cite{DGK17}). In this case those representations are also called $\h^{p,q}$-convex cocompact.

\bibstyle{plain}
\bibliography{refs}

\begin{thebibliography}{10}

\bibitem{barbot2013deformations}
Thierry Barbot.
\newblock {Deformations of Fuchsian AdS representations are quasi-Fuchsian}.
\newblock {\em Journal of Differential Geometry}, 101(1):1 -- 46, 2015.

\bibitem{beyrer2023mathbbhpqconvex}
Jonas Beyrer and Fanny Kassel.
\newblock $\mathbb{H}^{p,q}$-convex cocompactness and higher higher teichm\"uller spaces, 2023.

\bibitem{chalumeau2024properquasihomogeneousdomainseinstein}
Adam Chalumeau and Blandine Galiay.
\newblock Proper quasi-homogeneous domains of the einstein universe, 2024.

\bibitem{CollierTholozanToulisse}
Brian Collier, Nicolas Tholozan, and J{\'e}r{\'e}my Toulisse.
\newblock {The geometry of maximal representations of surface groups into $\mathrm{SO}_{0}(2,n)$}.
\newblock {\em Duke Mathematical Journal}, 168(15):2873 -- 2949, 2019.

\bibitem{DGK17}
Jeffrey Danciger, François Gueritaud, and Fanny Kassel.
\newblock Convex cocompact actions in real projective geometry.
\newblock {\em Annales Scientifiques de l'Ecole Normale Superieure}.

\bibitem{danciger2017convex}
Jeffrey Danciger, François Guéritaud, and Fanny Kassel.
\newblock Convex cocompactness in pseudo-riemannian hyperbolic spaces.
\newblock {\em Geometriae Dedicata}, 192, 02 2018.

\bibitem{Gerochsplittingtheorem}
Robert {Geroch}.
\newblock {Domain of Dependence}.
\newblock {\em Journal of Mathematical Physics}, 11(2):437--449, February 1970.

\bibitem{Zimmer}
Wouter~Van Limbeek and Andrew Zimmer.
\newblock {Rigidity of convex divisible domains in flag manifolds}.
\newblock {\em Geometry and Topology}, 23(1):171 -- 240, 2019.

\bibitem{mess2007lorentz}
Geoffrey Mess.
\newblock Lorentz spacetimes of constant curvature.
\newblock {\em Geometriae Dedicata}, 126, 07 2007.

\bibitem{schwartz1969nonlinear}
J.T. Schwartz.
\newblock {\em Nonlinear Functional Analysis}.
\newblock Notes on mathematics and its applications. Gordon and Breach, 1969.

\bibitem{seppi2023completemaximalsubmanifoldspseudohyperbolic}
Andrea Seppi, Graham Smith, and Jérémy Toulisse.
\newblock On complete maximal submanifolds in pseudo-hyperbolic space, 2023.

\bibitem{smai22}
Rym Smai.
\newblock Globally hyperbolic spatially compact maximal conformally flat spacetimes arising from anosov representations.
\newblock {\em Geometriae Dedicata}, 2022.

\bibitem{topologyoffiberbundles}
Norman Steenrod.
\newblock {\em The topology of fibre bundles}.
\newblock Princeton landmarks in mathematics and physics. Princeton University Press, 1999.

\bibitem{largescalestructure}
G.~F. R.~Ellis Stephen W.~Hawking.
\newblock {\em The large scale structure of space-time}.
\newblock Cambridge Monographs on Mathematical Physics. Cambridge University Press, 1975.

\bibitem{Thurston2}
W.~P. Thurston.
\newblock Existence of codimension-one foliations.
\newblock {\em Annals of Mathematics}, 104(2):249--268, 1976.

\bibitem{Thurston1974}
William Thurston.
\newblock The theory of foliations of codimension greater than one.
\newblock {\em Commentarii mathematici Helvetici}, 49:214--231, 1974.

\bibitem{Zimmer2}
Andrew~M. Zimmer.
\newblock Proper quasi-homogeneous domains in flag manifolds and geometric structures.
\newblock {\em Annales de l'Institut Fourier}, 68(6):2635--2662, 2018.

\end{thebibliography}

\end{document}